\documentclass[12pt]{article}
\usepackage{amsmath}
\usepackage{amsthm}
\usepackage{amsfonts}
\usepackage[labelformat=simple]{subcaption}

\usepackage{enumitem}
\setlist[enumerate]{label={\upshape(\alph*)}}

\usepackage{setspace}

\usepackage{color}

\usepackage{url}

\usepackage{tikz}
\usetikzlibrary{calc}
\usetikzlibrary{arrows}
\tikzstyle{vertex}=[circle, draw, inner sep=0pt, minimum size=5pt,fill=black]
\newcommand{\vertex}{\node[vertex]}
\tikzstyle{hollowvertex}=[circle, draw, inner sep=0pt, minimum size=5pt]
\newcommand{\hollowvertex}{\node[hollowvertex]}
\tikzstyle{namedvertex}=[circle, draw, inner sep=1pt, minimum size=12pt]

\tikzstyle{namedvertexabove}=[circle, draw, inner sep=1pt, minimum size=12pt,above]

\tikzstyle{phantomvertex}=[circle, draw, inner sep=0pt, minimum size=4pt,color=white]

\usetikzlibrary{decorations.markings}
\tikzset{->-/.style={decoration={
  markings,
  mark=at position #1 with {\arrow{>}}},postaction={decorate}}}
\newcommand{\od}{\mbox{od}}
\newcommand{\id}{\mbox{id}}
\newcommand{\kb}{\bar{\kappa}}
\newcommand{\kbm}{\bar{\kappa}_{\max}}
\newcommand{\Km}{K_{\max}}
\newcommand{\lb}{\bar{\lambda}}
\newcommand{\lbm}{\bar{\lambda}_{\max}}

\newtheorem{theorem}{Theorem}[section]

\newtheorem{lemma}[theorem]{Lemma}
\newtheorem{corollary}[theorem]{Corollary}
\newtheorem{observation}[theorem]{Observation}
\usetikzlibrary{decorations.markings}
\tikzset{->-/.style={decoration={markings,
  mark=at position #1 with {\arrow{>}}},postaction={decorate}}}

\theoremstyle{definition}

\newtheorem{conjecture}[theorem]{Conjecture}

\newtheorem{problem}[theorem]{Problem}

\newtheorem{examplex}[theorem]{Example}
\newenvironment{example}
  {\pushQED{\qed}\examplex}
  {\popQED\endexamplex}
\theoremstyle{remark}

\newcommand{\KBM}{\mathop{\bar{\kappa}_{\max}}}

\newcommand{\KAPPASUM}{\theta}

\begin{document}

\title{The maximum average connectivity among all orientations of a graph}
\author{Roc\'{i}o M.\ Casablanca\\
Universidad de Sevilla\\
\small \url{rociomc@us.es}\\
Peter Dankelmann\\ University of Johannesburg\\
\small \url{pdankelmann@uj.ac.za}\\
Wayne Goddard \\ Clemson University and University of Johannesburg\\
\small \url{goddard@clemson.edu}\\
Lucas Mol and Ortrud Oellermann\thanks{Supported by an NSERC Grant CANADA, Grant number RGPIN-2016-05237}\\
University of Winnipeg\\
\small \url{l.mol@uwinnipeg.ca,
 o.oellermann@uwinnipeg.ca}}
\date{}

\maketitle

\begin{abstract}
\noindent
For distinct vertices $u$ and $v$ in a graph $G$, the {\em connectivity} between $u$ and $v$, denoted $\kappa_G(u,v)$, is the maximum number of internally disjoint $u$--$v$ paths in $G$. The {\em average connectivity} of $G$, denoted $\overline{\kappa}(G),$ is the average of $\kappa_G(u,v)$ taken over all unordered pairs of distinct vertices $u,v$ of $G$. Analogously, for a directed graph $D$, the {\em connectivity} from $u$ to $v$, denoted $\kappa_D(u,v)$, is the maximum number of internally disjoint directed $u$--$v$ paths in $D$.  The {\em average connectivity} of $D$, denoted $\overline{\kappa}(D)$, is the average of $\kappa_D(u,v)$ taken over all ordered pairs of distinct vertices $u,v$ of $D$. An {\em orientation} of a graph $G$ is a directed graph obtained by assigning a direction to every edge of $G$. For a graph $G$, let $\overline{\kappa}_{\max}(G)$ denote the maximum average connectivity among all orientations of $G$. In this paper we obtain bounds for $\overline{\kappa}_{\max}(G)$ and for the ratio $\overline{\kappa}_{\max}(G)/\overline{\kappa}(G)$ for all graphs $G$ of a given order and in a given class of graphs. Whenever possible, we demonstrate sharpness of these bounds. This problem had previously been studied for trees.  We focus on the classes of cubic $3$-connected graphs, minimally $2$-connected graphs, $2$-trees, and maximal outerplanar graphs.
\end{abstract}

\section{Introduction}
In this article, a graph is finite, loopless, and contains no multiple edges.  An {\em orientation} of a graph $G$ is a directed graph obtained by assigning a direction to every edge of $G$.  Connectedness properties of orientations of graphs have been studied in a variety of different settings.  Probably the most well-known and oldest result in this area is Robbins' Theorem~\cite{Robbins1939}, which states that every $2$-edge-connected graph has a {\em strong orientation}, i.e., an orientation with the property that for every pair $u,v$ of distinct vertices of the graph, there is both a directed $u$--$v$ path and a directed $v$--$u$ path.  Nash-Williams~\cite{Nash-williams1960} extended this result by showing that every $2k$-edge-connected graph has a {\em strongly} $k$-{\em edge-connected} orientation, i.e., an orientation for which there exist $k$ edge-disjoint paths from $u$ to $v$ for every ordered pair $(u,v)$ of distinct vertices of the graph.  Mader~\cite{Mader1978} also established this result using his so-called lifting theorem. In light of these results on edge connectivity, it is natural to ask what can be said about the connectivity of an orientation of a graph in terms of its connectivity.  Thomassen~\cite{Thomassen1989} proposed the following conjecture.

\begin{conjecture}\label{thomassenconj}
For every positive integer $k$, there exists a smallest positive integer $f(k)$ such that every $f(k)$-connected graph has a strongly $k$-connected orientation.
\end{conjecture}

More recently, Thomassen~\cite{Thomassen2014} established the following necessary and sufficient conditions that guarantee that a graph admits a strongly $2$-connected orientation.

\begin{theorem}{\em \cite{Thomassen2014}}
A graph $G$ has a strongly $2$-connected orientation if and only if $G$ is $4$-edge-connected, and every vertex-deleted subgraph of $G$ is $2$-edge-connected. \label{Thomassen}
\end{theorem}

It follows that if $G$ is $4$-connected, then $G$ has a strongly $2$-connected orientation, confirming Conjecture~\ref{thomassenconj} in the case $k=2$.  Conjecture~\ref{thomassenconj} remains open for $k\ge 3$.  Durand de Gevigney~\cite{DurandDegevigney2012} showed that for each $k \ge 3$, the problem of deciding whether a given graph has a strongly $k$-connected orientation is NP-complete.

\bigskip

Instead of trying to find the largest $k$ for which a given graph has a strongly $k$-connected orientation, we focus on finding orientations for which the {\em average connectivity}, that is, the average of the connectivities between all ordered pairs of vertices, is maximized.

For distinct vertices $u$ and $v$ in a graph $G$, the {\em connectivity} between $u$ and $v$, denoted $\kappa_G(u,v)$, or $\kappa(u,v)$ if $G$ is clear from context, is the maximum number of internally disjoint $u$--$v$ paths in $G$. If $u$ and $v$ are not adjacent in $G$, then Menger's Theorem states that $\kappa_G(u,v)$ is also equal to the minimum number of vertices whose removal separates $u$ and $v$ in $G$.  It is also well-known that the connectivity $\kappa(G)$ of a graph $G$ is the minimum of $\kappa_G(u,v)$ taken over all unordered pairs of distinct vertices $u,v$ of $G$.  See~\cite{OellermannChapter2013} for more details.

The {\em average connectivity} of $G$, denoted $\kb(G),$ is the average of $\kappa_G(u,v)$ taken over all unordered pairs of distinct vertices $u,v$ of $G$.   If $G$ has order $n$, then
\[\kb(G)=\frac{1}{\binom{n}{2}}{\sum_{\{u,v\} \subseteq V(G)} \kappa_G(u,v)}.\]
This parameter was introduced by Beineke et al.~\cite{BeinekeOellermannPippert2002} as a more refined measure of the connectedness of  a graph than the connectivity.  Bounds on the average connectivity of a graph in terms of various graph parameters were given by Dankelmann and Oellermann~\cite{DankelmannOellermann2003}.

For a directed graph $D$ with distinct vertices $u$ and $v$, the {\em connectivity} from $u$ to $v$, denoted $\kappa_D(u,v)$, is the maximum number of internally disjoint directed $u$--$v$ paths in $D$.  The {\em average connectivity} of $D$, denoted $\kb(D)$, is defined as the average of $\kappa_D(u,v)$ taken over all ordered pairs of distinct vertices $u,v$ of $D$.  If $D$ has order $n$, then
\[
\kb(D)=\frac{1}{n(n-1)}\sum_{\substack{(u,v)\\u,v\in V(D), u\neq v}}\kappa_D(u,v).
\]
The average connectivity of digraphs was first introduced by Henning and Oellermann~\cite{HenningOellermann2004}.

In this article, we are concerned primarily with the maximum average connectivity among all orientations of a given graph $G$, denoted $\kbm(G)$.  This parameter was also introduced by Henning and Oellermann~\cite{HenningOellermann2004}.  If $D$ is an orientation of $G$ such that $\kbm(G) =\kb(D)$, then we say that $D$ is an {\em optimal orientation} of $G$.  Note that optimal orientations need not be unique; a graph may have many different optimal orientations.  Henning and Oellermann gave the following asymptotically sharp bound on $\kbm(T)$ for any tree $T$.\begin{footnote}{The family of trees described in~\cite{HenningOellermann2004} for which the lower bound is asymptotically sharp can be obtained as follows.  For a given $t \ge 1$, take three copies of $K_{1,t}$, and identify a leaf from each copy of $K_{1,t}$ in a single vertex (which will have degree $3$).  Let $T_{3t+1}$ be such a tree.  So if $n=3t+1$, then $\kbm(T_{3t+1}) = \frac{2n^2+14n-43}{9n(n-1)}$.  We point out that it was incorrectly stated in~\cite{HenningOellermann2004} that for a tree $T$ of order $n \ge 3$,  we have $\kbm(T) \ge \frac{2n^2+14n-43}{9n(n-1)}$. This inequality holds for $n \ge 34$.   However, for $n<34$, we have
\begin{align*}
\kbm(K_{1,n-1})&=\frac{\lfloor\frac{n-1}{2}\rfloor\lceil\frac{n-1}{2}\rceil+(n-1)}{n(n-1)}< \frac{2n^2+14n-43}{9n(n-1)}.
\end{align*}
For $n<  34$, the stars are in fact the extremal trees.  That is, for $n<34$, one can show that $\kbm(K_{1,n-1}) \le \kbm(T)$ for every tree $T$ of order $n$.}
\end{footnote}

\begin{theorem} {\em \cite{HenningOellermann2004}} \label{henningoellermann}
If $T$ is a tree of order $n \ge 3$, then
\[\tfrac{2}{9} < \kbm(T) \le \tfrac{1}{2}.\]
\end{theorem}

We obtain bounds on $\kbm(G)$ for all graphs $G$ of a given order and belonging to a given class of graphs.  Whenever possible, we demonstrate that these bounds are (asymptotically) sharp.  We study these problems for two classes of graphs that are, in some sense, generalizations or extensions of trees: minimally $2$-connected graphs (trees are minimally $1$-connected), and maximal outerplanar graphs (which are known to be $2$-trees).

We are also interested in the value of the ratio $\kbm(G)/\kb(G)$ for a given graph $G$, which is in some sense a measure of how well the overall level of connectedness of a graph can be preserved under orientation.  Naively, one might expect the ratio $\kbm(G)/\kb(G)$ to be close to $1/2$, since one might hope for an orientation of $G$ in which a collection of $\kappa_G(u,v)$ internally disjoint paths are all directed one way or the other, for every pair of distinct vertices $u,v$.  But the directed $u$--$v$ paths need not be internally disjoint from the directed $v$--$u$ paths, meaning that $\kbm(G)/\kb(G)$ can be much larger than $1/2$; we show that it can be arbitrarily close to $1$.  It is also true that $\kbm(G)/\kb(G)$ can be much smaller than $1/2$.  For example, we have already mentioned that there are trees $T$ such that $\kbm(T)$ is arbitrarily close to $2/9$.  Since $\kb(T)=1$ for every tree $T$, it follows immediately that $\kbm(T)/\kb(T)$ can be arbitrarily close to $2/9$.


We now describe our main contributions, and the layout of the article. In Section~\ref{Preliminaries}, we present some terminology, and some straightforward bounds on $\kbm(G)$ and $\kbm(G)/\kb(G)$ for every graph $G$.  We briefly consider the edge connectivity analogue of $\kbm(G)/\kb(G)$ in order to highlight a stark contrast between connectivity and edge connectivity in this setting.

In Section~\ref{OddRegular}, we show that if $G$ is an $r$-regular graph of order $n$ for odd $r$, then
\[
\KBM(G) \le \tfrac{r-1}{2} + \tfrac{n}{4(n-1)},
\]
and that this bound is sharp.  We then focus on cubic $3$-connected graphs.  If $G$ is a cubic $3$-connected graph, then certainly $\kbm(G)\geq 1$ by Robbins' Theorem.  We demonstrate that this lower bound is asymptotically tight by describing a sequence of cubic $3$-connected graphs for which the average connectivity approaches $1$. This shows that one cannot guarantee significantly more `connectedness' in an optimal orientation of a $3$-connected graph than in an optimal orientation of a $2$-edge-connected graph.

In Section~\ref{Minimally2Connected}, we show that if $G$ is a minimally $2$-connected graph of order $n$, then
\[
1\leq \kbm(G) <\tfrac{5}{4}.
\]
While the lower bound is sharp, we suspect that the upper bound can be improved.  We also show that for every minimally $2$-connected graph $G$,
\[
\tfrac{4}{9}< \tfrac{\kbm(G)}{\kb(G)}< \tfrac{5}{8}.
\]
Although we are unable to show that these bounds are sharp, we do find sequences of minimally $2$-connected graphs for which the ratio $\kbm(G)/\kb(G)$ approaches $\frac{25}{54}$ and $\frac{9}{16}$, respectively.  One of these constructions uses the sequence of cubic $3$-connected graphs described in Section~\ref{OddRegular}.

Finally, in Section~\ref{MOPs}, we show that if $G$ is a maximal outerplanar graph, then
\[
\kbm(G) \leq \tfrac{3}{2} + \tfrac{n-5}{n^2-n},
\]
and that this bound is asymptotically sharp.  We conjecture that if $G$ is a maximal outerplanar graph of order at least $4$, then $\kbm(G)\geq 19/18$.  We give an example which demonstrates that if this conjecture is true, then the bound is asymptotically sharp.


\section{General bounds}\label{Preliminaries}

We begin with some notation and terminology.  The {\em total connectivity} of a  graph $G$ is the sum of the connectivities of all unordered pairs of distinct vertices of $G$, and is denoted by $K(G)$.  Evidently, if $G$ has order $n$, then $K(G)=\binom{n}{2}\kb(G).$ If $u$ and $v$ are distinct vertices of $G$, then $\kappa(u,v) \le \min\{\deg(u) + \deg(v)\}$. Thus, if $d_1, d_2, \ldots, d_n$ is the degree sequence of $G$, then $K(G) \le \sum_{1 \le i < j \le n}\min\{d_i,d_j\}$. In this case, we call
\[
P(G)=P(d_1,d_2, \ldots, d_n) =\sum_{1 \le i < j \le n}\min\{d_i,d_j\}
\]
the {\em potential} of $G$.

If $D$ is a digraph, and $u,v \in V(D)$ are distinct, then we let $\theta_D(u,v) = \kappa_D(u,v) + \kappa_D(v,u)$. If $D$ is clear from context, then the subscript will be omitted. We also refer to $\theta(u,v)$ as the $\theta$ {\em value} for $u$ and $v$. For every pair $u,v$ of distinct vertices of $D$, we have $\theta(u,v) = \kappa(u,v) + \kappa(v,u) \le \min\{\od(u),\id(v)\} + \min\{\od(v), \id(u)\}$. We call
\[
P(D)=\sum_{\{u,v\} \subseteq V(D)} \min\{\od(u),\id(v)\} + \min\{\od(v), \id(u)\}
\]
the {\em potential} of $D$.

The {\em total connectivity} of a digraph $D$, denoted $K(D)$, is the sum of the connectivities of all ordered pairs of distinct vertices of $D$, or equivalently, the sum of the $\theta$ values of all unordered pairs of distinct vertices of $D$.  If $D$ has order $n$, then $K(D)=n(n-1)\kb(D).$

For a graph $G$, the notation $K_{\max}(G)$ denotes the maximum total connectivity among all orientations of $G$. The potential of $G$ provides the following useful upper bound on the total connectivity of any orientation of $G$, and hence on $K_{\max}(G)$.

\begin{observation}\label{degree_upperbd}
If $D$ is an orientation of a graph $G$, then
\[
K(D)=\sum_{\{u,v\} \subseteq V(D)} \theta(u,v) \le P(D) \le P(G).
\]
\end{observation}


Let $D$ be an orientation of $G$.  We call a pair of vertices $u$ and $v$ of $D$ {\em full} if
$\KAPPASUM_D(u,v)= \min \{ \deg_G(u), \deg_G(v) \}$.  We say that $D$ is {\em saturated}
if every pair of distinct vertices is full, i.e., if $K(D)=P(G).$

\smallskip

We now present some preliminary results.  We begin with some straightforward bounds on $\kbm(G)$.  From Robbins' theorem~\cite{Robbins1939}, we know that if $G$ is a $2$-edge-connected graph, then $G$ has a strong orientation.  This gives the following.

\begin{theorem}\label{Robbins}
If $G$ is a $2$-edge-connected graph, then $\kbm(G)\geq 1.$
\end{theorem}

\noindent
We remark, however, that we do not know whether every optimal orientation of a $2$-edge-connected graph must be strong.

If $D$ is an orientation of a graph of order $n$, then $\theta(u,v) \le n-1$ for all pairs $u,v$ of vertices of $D$. This gives the following bound, first noted by Henning and Oellermann.

\begin{theorem}{\em \cite{HenningOellermann2004}}
If $G$ is a graph of order $n$, then $\kbm(G)\leq \frac{n-1}{2}.$
\end{theorem}

\noindent
This bound is achieved, for example, if $n$ is odd and $G \cong K_n$~\cite{HenningOellermann2004}.

We now turn to bounds on the ratio $\kbm(G)/\kb(G)$.

\begin{theorem}\label{GeneralRatioBound}
For every graph $G$, we have $\kbm(G)/\kb(G)\leq 1$, and this bound is asymptotically sharp.
\end{theorem}

\begin{proof}
If $D$ is an orientation of $G$, then $\theta_D(u,v) \le 2\kappa_G(u,v)$ for all pairs $u,v$ of vertices of $G$.  So $\kbm(G)/\kb(G) \le 1$.

To see that this bound is asymptotically sharp, let $F_{2n}$ be the lexicographic product $P_n \circ (2K_1)$, i.e., the graph obtained from two disjoint paths $P:v_1v_2 \ldots v_n$ and $Q:u_1u_2 \ldots u_n$ by adding the edges  $v_iu_{i+1}$ and $u_iv_{i+1}$ for $1 \le i < n$ (see Figure~\ref{large_ratio_family}). The graph $F_{2n}$ has $n-2$ pairs $u,v$ such that ${\kappa}(u,v)=4$, namely those pairs $u_i,v_i$ for $2 \le i \le n-1$, and $4(n-3)$ pairs $u,v$ such that ${\kappa}(u,v)=3$, namely those pairs of adjacent vertices of degree four.  For all remaining pairs $u,v$ of $F_{2n}$, we see that $\kappa_{F_{2n}}(u,v)=2$. Thus  $\lim_{n \rightarrow \infty}\kb(F_{2n})=2$.

\begin{figure}[htb]
\centering{
\begin{tikzpicture}
\tikzstyle{every path}=[color =black, line width = 0.3 pt, > = triangle 45]
\pgfmathsetmacro{\n}{4}
\pgfmathtruncatemacro{\m}{\n-1}
\foreach \x in {1,...,\n}
{
\hollowvertex (\x) at (\x,0) {};
\node[below] at (\x,0) {$v_{\x}$};
}
\foreach \x in {1,...,\n}
{
\pgfmathtruncatemacro{\a}{\n+\x}
\hollowvertex (\a) at (\x,1) {};
\node[above] at (\x,1) {$u_{\x}$};
}
\foreach \x in {1,...,\m}
{
\pgfmathtruncatemacro{\a}{\n+\x}
\pgfmathtruncatemacro{\b}{\n+\x+1}
\pgfmathtruncatemacro{\c}{1+\x}
\path
(\x) edge[->-=0.9] (\b)
(\a) edge[->-=0.9] (\c)
(\b) edge[->-=0.9] (\a)
(\c) edge[->-=0.9] (\x);}
\hollowvertex (a) at (5.5,1) {};
\node[above] at (5.5,1) {$u_{n-1}$};
\hollowvertex (b) at (5.5,0) {};
\node[below] at (5.5,0) {$v_{n-1}$};
\hollowvertex (c) at (6.5,1) {};
\node[above] at (6.5,1) {$u_{n}$};
\hollowvertex (d) at (6.5,0) {};
\node[below] at (6.5,0) {$v_{n}$};
\path
(b) edge[->-=0.9] (c)
(a) edge[->-=0.9] (d)
(c) edge[->-=0.9] (a)
(d) edge[->-=0.9] (b);
\node at (4.75,0) {$\cdots$};
\node at (4.75,1) {$\cdots$};

\end{tikzpicture}}
\caption{The orientation $D_{2n}$ of the graph $F_{2n}$} \label{large_ratio_family}
\end{figure}
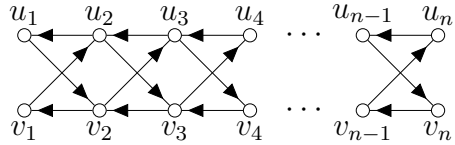

We now describe an orientation $D_{2n}$ of $F_{2n}$ with the property that $\lim_{n \rightarrow \infty} \tfrac{\kb(D_{2n})}{\kb(F_{2n})}=1$ (see Figure~\ref{large_ratio_family}). Orient $P$ from $v_n$ to $v_1$ and $Q$ from $u_n$ to $u_1$. The edges $u_iv_{i+1}$ and $v_iu_{i+1}$ are oriented as $(u_i,v_{i+1})$ and $(v_i,u_{i+1})$ for $1 \le i < n$.  For all pairs $u,v$ of vertices of degree $4$, we have $\theta_{D_{2n}}(u,v) =4$.  Since asymptotically almost all pairs of vertices have degree $4$, our claim follows. Hence, we have $\lim_{n \rightarrow \infty} \tfrac{\kbm(F_{2n})}{\kb(F_{2n})}=1$.
\end{proof}


While Theorem~\ref{GeneralRatioBound} gives an asymptotically sharp upper bound on the ratio $\kbm(G)/\kb(G)$ for every graph $G$, it is an open problem to determine whether there is a positive constant $c$ such that for every connected graph $G$, we have $\kbm(G)/\kb(G) \ge c$.

We briefly consider the edge connectivity analogue of $\kbm(G)/\kb(G)$ for the sake of comparison.  Let $\lambda(u,v)$, $\lb(G)$, and $\lbm(G)$ denote the edge connectivity analogues of $\kappa(u,v)$, $\kb(G)$, and $\kbm(G)$, respectively.  While the ratio $\kbm(G)/\kb(G)$ can be arbitrarily close to $1$, this is not the case for $\lbm(G)/\lb(G).$

\begin{theorem}
For every graph $G$, we have $\lbm(G)/\lb(G)\leq 1/2$.
\end{theorem}

\begin{proof}
Let $G$ be a graph with distinct vertices $u$ and $v$.  Let $D$ be any orientation of $G$.  There is a set of $\lambda_G(u,v)$ edges whose removal leaves a component $C_u$ containing $u$, and a component $C_v$ containing $v$.  Suppose that exactly $k$ of the edges are oriented from a vertex of $C_u$ to a vertex of $C_v$.  Every directed $u$-$v$ path in $D$ must use at least one of these edges, so $\lambda_D(u,v)\leq k$.  Similarly, $\lambda_D(v,u)\leq \lambda_G(u,v)-k$.  Thus, we have $\lambda_D(u,v)+\lambda_D(v,u)\leq \lambda_G(u,v)$ for every pair of distinct vertices $u,v$ of $G$, and the theorem statement follows immediately.
\end{proof}

In proving that every $2k$-edge-connected graph has a strongly $k$-edge-connected orientation, Nash-Williams actually proved the stronger result that for every graph $G$, there is an orientation $D$ of $G$ such that for every ordered pair $(u,v)$ of distinct vertices of $D$, there exist at least $\left\lfloor \lambda_G(u,v)/2 \right\rfloor$ edge-disjoint directed $u$-$v$ paths in $D$~\cite{Nash-williams1960}.  It follows immediately that $\lbm(G) \ge \lb(D) \ge \frac{\lb(G)-1}{2}$, or equivalently,
\[
\frac{\lbm(G)}{\lb(G)}\ge \frac{1}{2}-\frac{1}{2\lb(G)}.
\]
We can easily obtain a constant lower bound on $\lbm(G)/\lb(G)$ in the case that $G$ is $2$-edge-connected, and we will see that this bound is asymptotically sharp in Section~\ref{CubicSection}.

\begin{theorem}\label{EdgeLowerBound}
If $G$ is a $2$-edge-connected graph, then $\lbm(G)/\lb(G)\geq 1/3.$
\end{theorem}
\begin{proof}
Let $G$ be a $2$-edge-connected graph.  By the result of Nash-Williams, there is an orientation $D$ of $G$ such that for every pair of distinct vertices $u,v$ of $G$, we have $\lambda_D(u,v)\geq \left\lfloor \lambda_G(u,v)/2 \right\rfloor$.  We claim that $\lambda_D(u,v)+\lambda_D(v,u)\geq \tfrac{2}{3}\lambda_G(u,v)$ for every pair of distinct vertices $u,v$ of $G$, from which the theorem statement easily follows.  If $\lambda_G(u,v)$ is even, then we have
\[
\lambda_D(u,v)+\lambda_D(v,u)\geq 2\cdot\left\lfloor \lambda_G(u,v)/2 \right\rfloor=\lambda_G(u,v).
\]
On the other hand, if $\lambda_G(u,v)$ is odd, then
\[
\lambda_D(u,v)+\lambda_D(v,u)\geq 2\cdot \left\lfloor \lambda_G(u,v)/2 \right\rfloor= \lambda_G(u,v)-1\geq \tfrac{2}{3}\lambda_G(u,v),
\]
where we used the fact that $G$ is $2$-edge-connected, meaning that $\lambda_G(u,v)\geq 3$ in this case.  This completes the proof of the claim, and hence the theorem.
\end{proof}

Overall, we see that the parameters $\lbm(G)/\lb(G)$ and $\kbm(G)/\kb(G)$ appear to behave rather differently in general.


\section{Odd regular graphs}\label{OddRegular}

In this section, we study bounds on the average connectivity of optimal orientations of $r$-regular $r$-connected graphs for odd $r$.  A graph $G$ is \emph{uniformly $r$-connected} if $\kappa(G) = \kb(G)=r$ (see \cite{BeinekeOellermannPippert2002}).  Clearly all $r$-regular $r$-connected graphs are uniformly $r$-connected.

Orientations of odd regular graphs are never saturated, due to the following elementary result.

\begin{observation}\label{InAndOut}
Let $D$ be an orientation of an $r$-regular graph.  If a pair of vertices $u$, $v$ in $D$ is full, then the out-degree of~$u$ equals the in-degree of $v$.
\end{observation}

\noindent
So, if $G\neq K_2$ is $r$-regular with saturated orientation
$D$, then $r$ is even and the orientation of $D$ is regular; that is, every vertex has in- and out-degree $r/2$.  For even $r$, there do exist $r$-regular $r$-connected graphs with saturated orientations (see~\cite{HenningOellermann2004, HenningOellermann2001}).

\subsection{An upper bound on \boldmath{$\kbm(G)$}}


\begin{theorem}\label{KBM_bd_odd_regular}
Let $r\geq 3$ be odd.  If $G$ is an $r$-regular graph on $n$ vertices, then
\[
\KBM(G) \le \frac{r-1}{2} + \frac{n}{4(n-1)}.
\]
\end{theorem}
\begin{proof}
Let $D$ be an arbitrary orientation of $G$.
Let $C$ be the number of full pairs in the orientation $D$.
Then the total connectivity of $D$ is at most $C+(r-1)n(n-1)/2$.
Let $A$ be the vertices of $G$ with larger in-degree than out-degree.
Then it follows from Observation~\ref{InAndOut} that a full pair contains
one vertex belonging to $A$ and one vertex not belonging to $A$. So the number of full pairs of $D$ is at most $n^2/4$.  Our result now follows. 
\end{proof}

We now describe, for every odd $r\geq 3$, a family of $r$-regular graphs achieving the upper bound of Theorem \ref{KBM_bd_odd_regular}.
Let $s, t \ge 2$ be integers, and let $H_{s,t}$ be the graph constructed as follows.
Take $2s$ sets of vertices $V_1, \ldots, V_{2s}$, each of size $t$.
For every odd $i$, join every vertex in $V_i$ to every vertex in $V_{i+1}$. For
every even $i$, add all edges between the vertices of $V_{i}$ and $V_{i+1}$ apart from a perfect matching (where subscripts are expressed modulo $2s$).
Then $H_{s,t}$ is a regular graph of degree $2t-1$.
For example, $H_{2,2}$ is the hypercube~$Q_3$, and  $H_{5,3}$ is shown in Figure \ref{H_53}.

\begin{figure}[h!]
\begin{center}
\begin{tikzpicture}
\tikzstyle{every path}=[color =black, line width = 0.3 pt, > = triangle 45]
\pgfmathsetmacro{\n}{5}
\pgfmathtruncatemacro{\m}{\n-1}
\pgfmathtruncatemacro{\p}{2*\n-1}
\foreach \x in {0,...,\m}
{
\pgfmathtruncatemacro{\a}{6*\x}
\pgfmathtruncatemacro{\b}{6*\x+1}
\pgfmathtruncatemacro{\c}{6*\x+2}
\pgfmathtruncatemacro{\d}{6*\x+3}
\pgfmathtruncatemacro{\e}{6*\x+4}
\pgfmathtruncatemacro{\f}{6*\x+5}
\hollowvertex (\a) at (360*\x/\n:1.5) {};
\hollowvertex (\b) at (360*\x/\n:2.25) {};
\hollowvertex (\c) at (360*\x/\n:3) {};
\vertex (\d) at (360*\x/\n+180/\n:1.5) {};
\vertex (\e) at (360*\x/\n+180/\n:2.25) {};
\vertex (\f) at (360*\x/\n+180/\n:3) {};
}
\foreach \x in {0,...,\p}
{
\pgfmathtruncatemacro{\ya}{3*\x}
\pgfmathtruncatemacro{\zaa}{mod(\ya+3,6*\n)}
\pgfmathtruncatemacro{\zab}{mod(\ya+4,6*\n)}
\pgfmathtruncatemacro{\yb}{3*\x+1}
\pgfmathtruncatemacro{\zba}{mod(\yb+3,6*\n)}
\pgfmathtruncatemacro{\zbb}{mod(\yb+4,6*\n)}
\pgfmathtruncatemacro{\yc}{3*\x+2}
\pgfmathtruncatemacro{\zca}{mod(\yc+3,6*\n)}
\pgfmathtruncatemacro{\zcb}{mod(\yc+1,6*\n)}
\path
(\ya) edge (\zaa)
(\ya) edge (\zab)
(\yb) edge (\zba)
(\yb) edge (\zbb)
(\yc) edge (\zca)
(\yc) edge (\zcb)
;
}
\foreach \x in {0,...,\m}
{
\pgfmathtruncatemacro{\y}{2*\x+1}
\pgfmathtruncatemacro{\ya}{3*\y}
\pgfmathtruncatemacro{\za}{mod(\ya+5,6*\n)}
\pgfmathtruncatemacro{\yb}{3*\y+1}
\pgfmathtruncatemacro{\zb}{mod(\yb+2,6*\n)}
\pgfmathtruncatemacro{\yc}{3*\y+2}
\pgfmathtruncatemacro{\zc}{mod(\yc+2,6*\n)}
\path
(\ya) edge (\za)
(\yb) edge (\zb)
(\yc) edge (\zc)
;
}
\end{tikzpicture}
\caption{The graph $H_{5,3}$} \label{H_53}
\end{center}
\end{figure}
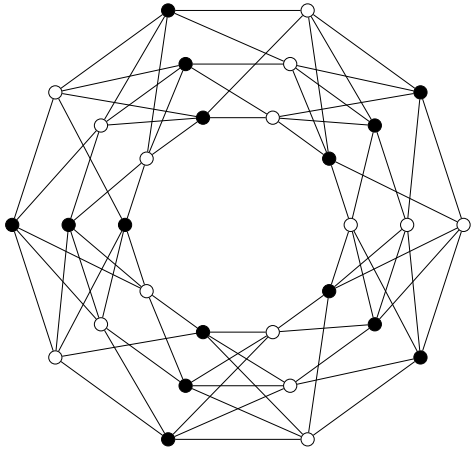

Now, let $D_{s,t}$ be the orientation of $H_{s,t}$ obtained by placing the sets $V_1,V_2,\dots,V_{2s}$ around a circle and orienting edges clockwise (see Figure~\ref{H_53}). To aid us in our discussions, we colour the vertices in the odd subscripted sets white, and those in the even subscripted sets black. In $D_{s,t}$, the white vertices have in-degree $t-1$ and out-degree $t$, while the black vertices have out-degree $t-1$ and in-degree $t$. We claim that there are $t$ internally disjoint paths from any white vertex to any black vertex; and there are $t-1$ internally disjoint paths from any white vertex to any other white vertex, as well as from any black vertex to any other black vertex, and from any black vertex to any white vertex.  It follows immediately that $\kb(D_{s,t})=\kbm(H_{s,t})$ matches the upper bound given by Theorem~\ref{KBM_bd_odd_regular}.

In the cubic case, the bound of Theorem \ref{KBM_bd_odd_regular} is also achieved by the M\"{o}bius ladders. Figure \ref{mobiusladder} shows a M\"{o}bius ladder of order $14$.

\begin{figure}[h!]
\begin{center}
\begin{tikzpicture}
\path[use as bounding box] (-0.3,-1) rectangle (7,2);
\tikzstyle{every path}=[color =black, line width = 0.3 pt, > = triangle 45]
\pgfmathsetmacro{\n}{7}
\pgfmathtruncatemacro{\m}{\n-1}
\pgfmathtruncatemacro{\p}{\n-2}
\foreach \x in {0,...,\m}
{
\pgfmathtruncatemacro{\y}{\x+\n}
\vertex (\x) at (\x,0) {};
\hollowvertex (\y) at (\x,1) {};
}
\foreach \x in {0,...,\m}
{
\pgfmathtruncatemacro{\y}{\x+\n}
\path
(\x) edge (\y)
;
}
\foreach \x in {0,...,\p}
{
\pgfmathtruncatemacro{\xa}{\x+1}
\pgfmathtruncatemacro{\ya}{\x+\n}
\pgfmathtruncatemacro{\yb}{\ya+1}
\path
(\x) edge (\xa)
(\ya) edge (\yb)
;
}
\pgfmathtruncatemacro{\q}{2*\n-1}
\path
(0) edge[out=-20,in=-30,looseness=1.3] (\q)
(\n) edge[out=20,in=30,looseness=1.3] (\m)
;

\end{tikzpicture}
\caption{The M\"{o}bius ladder of order 14} \label{mobiusladder}
\end{center}
\end{figure}
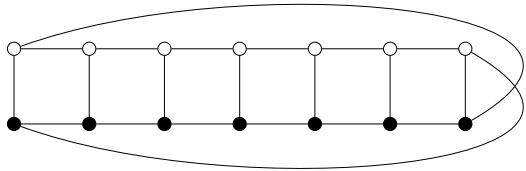


\subsection{The average connectivity in 3-connected cubic graphs}\label{CubicSection}

By Theorem~\ref{Thomassen}, we have $\kbm(G) \ge 2$ for every $4$-connected $4$-regular graph $G$.  In contrast, for a $3$-connected cubic graph $G$, we show that $\kbm(G)$ can be arbitrarily close to $1$.  This demonstrates that the lower bound of Theorem~\ref{Robbins} is asymptotically sharp even for $3$-connected cubic graphs.  Additionally, we see that the ratio $\kbm(G)/\kb(G)$ can be arbitrarily close to $1/3$ for a $3$-connected cubic graph $G$ (and it cannot be smaller than $1/3$).

We begin by considering triangles in cubic graphs.  Let $T$ be a triangle in a cubic graph $G$. In an orientation of the graph $G$ we define a vertex $u$ of $T$ as \emph{bad} if $\theta(u,v) \le 2$ for all vertices $v \in V(G)-V(T)$.

\begin{lemma} \label{triangles_have_bad_vertices}
Let $G$ be a cubic graph and let $T$ be a triangle in  $G$. Then in any orientation of $G$, at least
one vertex of $T$ is bad.
\end{lemma}
\begin{proof}
Let $D$ be an orientation of $G$. Consider the three arcs incident with vertices of $T$, but not belonging to $T$.  Suppose that $s$ of them are oriented away from $T$. In order to be in a full pair with a vertex not in $T$, a vertex of $T$ must have out-degree $s$. This is not possible for all vertices of $T$.
\end{proof}

Now, given a graph $G$, the graph $I(G)$, which we will call the \emph{inflation} of $G$ (see~\cite{Chvatal1973}), is defined as the line graph of the subdivision of $G$. In the case that $G$ is a cubic graph, the process of constructing $I(G)$ is sometimes referred to as making the \emph{wye-delta replacement} at each vertex of $G$. For a cubic graph $G$, the graph $I(G)$ is obtained as follows.  For every vertex $v \in V(G)$, let $T_v$ be a triangle with vertices $v_x,$ $v_y,$ and $v_z$, where $x,$ $y,$ and $z$ are the neighbours of $v$ in $G$.   The inflation $I(G)$ is obtained from the disjoint union $\cup_{v\in V(G)}T_v$ by joining $v_x$ with $x_v$, for every edge $vx$ of $G$. The inflation of $K_4$ is shown in Figure \ref{inflation}.

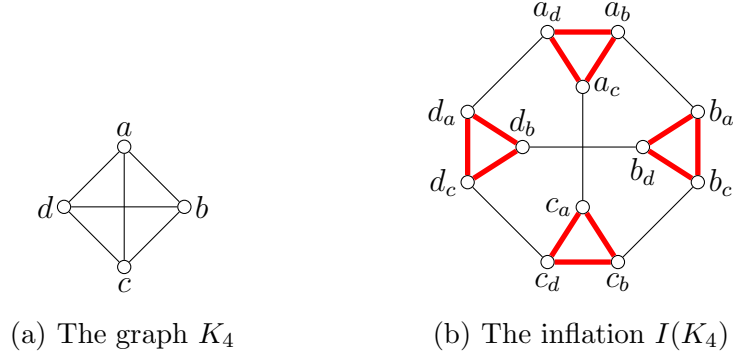
\begin{figure}[htb]
\centering
\begin{subfigure}[t]{0.4\textwidth}
\centering
\begin{tikzpicture}[scale=0.8]
\foreach \x in {0,1,2,3}
{
\hollowvertex (\x) at (\x*90:1){};
}
\node[above] at (1) {$a$};
\node[right] at (0) {$b$};
\node[below] at (3) {$c$};
\node[left] at (2) {$d$};
\path
(0) edge (1)
(0) edge (2)
(0) edge (3)
(1) edge (2)
(1) edge (3)
(2) edge (3)
;
\end{tikzpicture}
\caption{The graph $K_4$}
\end{subfigure}%
\begin{subfigure}[t]{0.4\textwidth}
\centering
\begin{tikzpicture}[scale=0.8]
\foreach \x in {0,1,2,3}
{
\hollowvertex (\x) at (\x*90:1){};
\pgfmathtruncatemacro{\y}{\x+4};
\hollowvertex (\y) at (\x*90+17:2){};
\pgfmathtruncatemacro{\z}{\x+8};
\hollowvertex (\z) at (\x*90-17:2){};
}
\node[right] at (1) {$a_c$};
\node[below] at (0) {$b_d$};
\node[left] at (3) {$c_a$};
\node[above] at (2) {$d_b$};
\node[above] at (9) {$a_b$};
\node[right] at (8) {$b_c$};
\node[below] at (11) {$c_d$};
\node[left] at (10) {$d_a$};
\node[above] at (5) {$a_d$};
\node[right] at (4) {$b_a$};
\node[below] at (7) {$c_b$};
\node[left] at (6) {$d_c$};
\path
(0) edge (2)
(1) edge (3)
(4) edge (9)
(5) edge (10)
(6) edge (11)
(7) edge (8)
;
\foreach \x in {0,1,2,3}
{
\pgfmathtruncatemacro{\y}{\x+4};
\pgfmathtruncatemacro{\z}{\x+8};
\path[color=red,line width=2 pt]
(\x) edge (\y)
(\x) edge (\z)
(\y) edge (\z)
;
}
\end{tikzpicture}
\caption{The inflation $I(K_4)$}
\end{subfigure}
\caption{The graphs $K_4$ and $I(K_4)$} \label{inflation}
\end{figure}

If $F$ is an orientation of $I(G)$, then the orientation of $G$ obtained by assigning $vx$ the orientation $(v,x)$ if $(v_x,x_v)$ is in $I(G)$, or $(x,v)$ if $(x_v,v_x)$ is in $I(G)$, is called the \emph{projection} of $F$ onto $G$. On the other hand, if $D$ is an orientation of $G$, and $F$ is an orientation of $I(G)$ with the property that $(v,x) \in E(D)$ implies $(v_x,x_v) \in E(F)$, then we say that $F$ is a \emph{lifting} of $D$ to $I(G)$.

\begin{lemma}\label{projected_orientations}
Let $G$ be a cubic graph. Let $F$ be an orientation of $I(G)$ and let $D$ be the projection of $F$ onto $G$.  If $D$ has exactly $C$ full pairs, then $F$ has at most $4C+2|V(G)|$ full pairs.
\end{lemma}
\begin{proof}
Suppose $T_x$ and $T_y$ are distinct triangles of $I(G)$. Then, by Lemma \ref{triangles_have_bad_vertices}, there are at most four full pairs of vertices of $I(G)$ that have one vertex in $T_x$ and the other in $T_y$.
If $P$ and $Q$ are (internally) disjoint $T_x$--$T_y$ paths in $I(G)$, then, apart from vertices in $T_x$ or $T_y$, the paths $P$ and $Q$ cannot contain vertices from the same triangle of $I(G)$. So $D$ has a pair of internally disjoint $x$--$y$ paths obtained from $P$ and $Q$ by contracting every triangle in $I(G)$ to the corresponding vertex in $G$.
It follows that $x$ and $y$ must be a full pair in $D$.  Thus, the number of full pairs of $F$ that are not contained in a single triangle is at most four times the number of full pairs of $D$.

Now it suffices to show that every triangle of $I(G)$ contains at most two full pairs of vertices. Suppose that $a$ and $b$ make up a full pair and belong to the same triangle $T_z$ of $I(G)$. If $\mathcal{F}_{a \rightarrow b}$ is a family of $\kappa_{I(G)}(a,b)$ internally disjoint $a$--$b$ paths, and $\mathcal{F}_{b \rightarrow a}$ is a family of $\kappa_{I(G)}(b,a)$ internally disjoint $b$--$a$ paths, then $\kappa_{I(G)}(a,b) +\kappa_{I(G)}(b,a)=3$, and every arc incident with $a$ and every arc incident with $b$ appears in $\mathcal{F}_{a \rightarrow b} \cup \mathcal{F}_{b \rightarrow a}$. So the arc between $a$ and $b$ constitutes one of the paths in this union, and a second path necessarily has length 2 and passes through the third vertex of $T_z$. The internal vertices of the third path in this union are necessarily not in $T_z$.  Moreover, if the arc incident with $a$ but not in $T_z$ is directed away from $a$, then the arc incident with $b$ and not in $T_z$ is directed towards $b$, and vice versa. This can only happen for two pairs of vertices in $T_z$.
\end{proof}

The previous lemma shows how orientations of $I(G)$ give rise to orientations of $G$ where connectedness properties between triangles in $I(G)$ translate to connectedness properties between the corresponding vertices of $G$. On the other hand, if $D$ is an orientation of $G$, and $F$ is a lifting of $D$ to $I(G)$ in which every triangle is oriented cyclically, then the next lemma establishes a connection between the connectedness properties of $D$ and those of $F$.

\begin{lemma} \label{lifted_orientations}
Let $G$ be a $3$-connected cubic graph, and let $D$ be a strong orientation of $G$. Let $F$ be a lifting of $D$ to $I(G)$ in which every triangle is oriented cyclically. If $D$ has exactly $C$ full pairs, then $F$ has $4C+2|V(G)|$ full pairs.
\end{lemma}

\begin{proof}
Note that the orientation $F$ of $I(G)$ is strong.  Consider a full pair $x$ and~$y$ in~$D$. Without loss of generality, assume that $x$ has out-degree $2$ and
$y$ has in-degree $2$. Let $a_1,a_2$ be the out-neighbours and $a_3$ the in-neighbour of $x$ in $D$. Let $b_1,b_2$ be the in-neighbours and $b_3$ the out-neighbour of $y$ in $D$. Then for $i\in \{1,2\}$ and $j \in \{1,2 \},$ there are two disjoint paths from $x_{a_i}$ to $y_{b_j}$. It follows that for two distinct triangles in $I(G)$, we have four
full pairs of vertices in $I(G)$ (where each pair contains a vertex from each of these two triangles).

Since every arc in $D$ is in a cycle, it follows that for $i\in \{1,2\}$ there is a path from
$x_{a_i}$ to $x_{a_3}$ that is internally disjoint from $T_x$. Thus for every triangle $T_x$ in $D$, there are two full pairs
with both vertices in $T_x$. This gives the desired result.
\end{proof}

Together, Lemma~\ref{projected_orientations} and Lemma~\ref{lifted_orientations} give the following.

\begin{theorem}\label{LiftStrong}
Let $G$ be a $3$-connected cubic graph of order $n$.  If $G$ has a strong optimal orientation $D$, and $F$ is a lifting of $D$ to $I(G)$ in which every triangle is oriented cyclically, then $F$ is a strong optimal orientation of $I(G)$. Moreover, \[\KBM(I(G)) = \kb(F) = 1+ \frac{4n(n-1)[\KBM(G)-1]+2n}{3n(3n-1)}.\]
\end{theorem}
\begin{proof}
\noindent Suppose $G$ has a  strong optimal orientation $D$. Then $\theta_D(u,v) \ge 2$ for all pairs $u,v$ of vertices of $D$. Also if $u,v$ is a full pair, then $\theta_D(u,v)-2 = 1$. Hence
\[
\sum_{\{u,v\} \subseteq V(G)}(\theta(u,v)-2)= K(D) -n(n-1)
\]
equals the number of full pairs in $D$. If $F$ is the orientation of $I(G)$ obtained by lifting $D$ to $I(G)$ and orienting each triangle cyclically, then, by Lemma \ref{lifted_orientations}, $F$ has $4(K(D)-n(n-1)) +2n$ full pairs. Thus
\[K(F) =3n(3n-1) + 4(K(D)-n(n-1)) +2n= 3n(3n-1) +4n(n-1)[\kb(D)-1]+2n.\]
 The result now follows.
 \end{proof}

It is readily seen that $K_4$ has a strong optimal orientation, and that $\KBM(K_4) = 4/3$.
Thus we obtain the following corollary to Theorem~\ref{LiftStrong}.

\begin{corollary}\label{inflation_kbm_goes_goes_to_1}
 $\displaystyle\lim_{k\rightarrow\infty}\kbm(I^k(K_4))=1.$
\end{corollary}

\noindent
We conclude that the bound of Theorem~\ref{Robbins} is asymptotically sharp even for $3$-connected cubic graphs, and that the ratio $\kbm(G)/\kb(G)$ can be made arbitrarily close to $1/3$ for a $3$-connected cubic graph $G$.

Finally, we note that for every pair of distinct vertices $u,v$ in a cubic graph, a collection of $u$-$v$ paths is internally disjoint if and only if it is edge disjoint.  Hence, if $G$ is a cubic graph, then $\lbm(G)/\lb(G)=\kbm(G)/\kb(G)$.  Therefore, it follows from Corollary~\ref{inflation_kbm_goes_goes_to_1} that the bound of Theorem~\ref{EdgeLowerBound} is asymptotically sharp.

\section{Orientations of minimally 2-connected graphs}\label{Minimally2Connected}

It is natural to ask which graphs $G$ satisfy $\kappa(G)=\bar{\kappa}(G)=k$ for some positive integer $k$. It was observed in~\cite{BeinekeOellermannPippert2002} that graphs with this property are {\em minimally} $k$-{\em connected}, i.e., $\kappa(G) = k$ and $\kappa(G-e) < k$ for all edges $e$ of $G$.  The minimally $1$-connected graphs are precisely the  trees, whose average connectivity is $1$.  However, the average connectivity of minimally $k$-connected graphs, for $k \ge 2$, need not be $k$.  Indeed, it has been shown that if $G$ is a minimally $2$-connected graph, then $2\leq \kb(G)< 9/4,$ and these bounds are asymptotically tight~\cite{CasablancaMolOellermann2018}. In this section, we determine bounds on $\kbm(G)$ and $\kbm(G)/\kb(G)$ for every minimally $2$-connected graph $G$.

We begin by stating several known results about minimally $k$-connected graphs. An edge $e$ of a $k$-connected graph $G$ is called $k$-{\em essential} if $\kappa(G-e) < \kappa(G)$.  Thus, a minimally $k$-connected graph is one in which every edge is $k$-essential. Mader~\cite{Mader1972} established the following structure theorem for the $k$-essential edges in a $k$-connected graph.

\begin{theorem}{\em \cite{Mader1972}}
If $G$ is a $k$-connected graph, and $C$ is a cycle of $G$ induced by  $k$-essential edges, then some vertex of $C$ has degree $k$ in $G$.
\end{theorem}

The following structure result for minimally $k$-connected graphs follows from the above.

\begin{corollary} {\em \cite{Mader1972}} \label{mader}
Let $G$ be a minimally $k$-connected graph, and let $F$ the subgraph induced by the vertices of degree exceeding $k$ in $G$. Then $F$ is a forest with at least two components.
\end{corollary}


Minimally $2$-connected graphs were characterized independently by Dirac~\cite{Dirac1967} and Plummer~\cite{Plummer1968}. A cycle $C$ of a graph $G$ is said to have a {\em chord} if there is an edge of $G$ that joins a pair of non-adjacent vertices from $C$. Plummer characterized the minimally $2$-connected graphs as follows.

\begin{theorem}{\em \cite{Plummer1968}} \label{plummer}
A $2$-connected graph $G$ is minimally $2$-connected if and only if no cycle of $G$ has a chord.
\end{theorem}

\subsection{Bounds on \boldmath{$\kbm(G)$}}

Here, we establish upper and lower bounds on $\KBM(G)$, where $G$ is a minimally $2$-connected graph.  Our first lemma concerns optimal orientations of connected graphs in general.

\begin{lemma}\label{NoSourceToSink}
Let $D$ be an optimal orientation of a connected graph $G$ of order at least $3$.  Then no arc of $D$ is oriented from a source to a sink.
\end{lemma}

\begin{proof}
Suppose otherwise that $D$ contains the arc $(u,v)$, where $u$ is a source and $v$ is a sink.  Then the only oriented path in $D$ that contains the arc $(u,v)$ is the  path of length $1$ from $u$ to $v$.  Let $D'$ be the orientation obtained by reversing the arc $(u,v)$ to obtain $(v,u)$.   Note that $\theta_{D'}(u,v)=\theta_D(u,v)$, and that $\theta_{D'}(x,y)\geq \theta_D(x,y)$ for any other pair of vertices in $G$, since no path between $x$ and $y$ used the arc $(u,v)$.  Further, there is a vertex $w\neq v$ such that either $(u,w)$ or $(w,v)$ is an arc in $D$, and there is an oriented path from $v$ to $w$ (or $w$ to $u$, respectively) in $D'$, while there was none in $D$.  This gives $\kb(D')>\kb(D),$ a contradiction.
\end{proof}

We now establish a structure result for every minimally $2$-connected graph $G$ of a given order for which $\kbm(G)$  has largest possible value.

\begin{lemma}\label{NoAdjacentDegree2}
Let $G$ be a minimally $2$-connected graph of order $n\geq 5$ such that $\kbm(G)$ is largest.  Then no two vertices of degree $2$ are adjacent in $G$.
\end{lemma}

\begin{proof} Let $G$ be as in the lemma statement.
Since $K_{2, n-2}$ has an orientation for which the total connectivity exceeds $n(n-1)$, we know that $G$ is not a cycle.  Suppose, towards a contradiction, that $G$ has two adjacent vertices of degree $2$, say $u$ and $v$.  Let $u'$ be the other neighbour of $u$ and let $v'$ be the other neighbour of $v$.  Since $G$ is $2$-connected and $n \ge 5$, we have $u' \ne v'$.  Further, since $G$ is minimally $2$-connected and is not a cycle, one can argue that $u'v'\not\in E(G)$ by using Theorem~\ref{plummer}. Let $G'$ be the graph obtained by deleting the edge $uv$ and adding the edges $u'v$ and $uv'$.  Using the fact that $u'v'\not\in E(G)$, it is straightforward to show that $G'$ is minimally $2$-connected.  We claim that $\kbm(G')>\kbm(G)$.

Let $D$ be an optimal orientation of $G$.  Suppose first that the path $P:u'uvv'$ is oriented from $u'$ to $v'$ or from $v'$ to $u'$ in $D$.  Let $D'$ be the orientation of $G'$ obtained from $D$ by deleting the arcs incident to $u$ and $v$, and adding the arcs $(u',u),$ $(u,v'),$ $(v',v),$ and $(v,u')$.    By straightforward arguments, we have $\theta_{D'}(u',v')=\theta_D(u',v')+1$, and $\theta_{D'}(x,y)\geq \theta_{D}(x,y)$ for every other unordered pair of vertices $x,y$. It follows that $\kb(D')>\kb(D).$

On the other hand, if the vertices of $P$ do not induce an oriented path in $D$, then either $u$ or $v$ is a source or a sink in $D$.  Suppose $u$ is a sink; the other cases are similar.  Then by Lemma~\ref{NoSourceToSink}, $v$ is not a source.  So  $(u',u)$, $(v,u)$, and $(v',v)$ are arcs in $D$.  Let $D'$ be the orientation of $G'$ obtained from $D$ by deleting the arc $(v,u)$ and adding the arcs $(v',u)$ and $(v,u')$.  Once again, one can verify that $\theta_{D'}(u',v')=\theta_D(u',v')+1$, and that $\theta_{D'}(x,y)\geq \theta_{D}(x,y)$ for every other unordered pair of vertices $x,y$.  It follows that $\kb(D')>\kb(D).$
\end{proof}

We are now ready to bound the value of $\kbm(G)$ for any minimally $2$-connected graph $G$.

\begin{theorem}\label{general_kappa_bar_max_bounds}
Let $G$ be a minimally $2$-connected graph of order $n \ge 3$.  Then
\[
1\leq \kbm(G)\leq 1+\tfrac{(n-3)^2}{4n(n-1)} <\tfrac{5}{4}.
\]
\end{theorem}

\begin{proof}
The lower bound follows immediately from Robbins' Theorem.  For the upper bound, let $G$ be a minimally $2$-connected graph of order $n\geq 3$ such that $\kbm(G)$ is largest.  Suppose that $G$ has $s$ vertices of degree at least $3$ and $n-s$ vertices of degree $2$.  Let $V_1=\{v_1,\dots,v_s\}$ be the set of vertices of degree at least $3$.   Let $d_i=\deg(v_i)$ for all $i\in\{1,\dots,s\}$.  From Corollary~\ref{mader}, $G[V_1]$ is a forest with at least two components.  So the subgraph induced by the vertices of degree at least $3$ has at most $s-2$ edges.   By Lemma~\ref{NoAdjacentDegree2}, every edge incident to a vertex of degree $2$ must also be incident to a vertex of degree at least $3$.  Thus we have
\[
\sum_{i=1}^s d_i=2(n-s)+2|E(G[V_1])|\leq 2(n-s)+2(s-2)=2n-4.
\]

Let $D$ be some orientation of $G$ and let $d$ and $0\leq r<s$ be the unique integers for which $2n-4=ds+r$.  By Observation \ref{degree_upperbd}, we have

\begin{align*}
K(D) \le P(G)& = 2\cdot\left[\tbinom{n}{2}-\tbinom{s}{2}\right]+P(d_1,d_2,\dots,d_s)\\
&\leq 2\cdot\left[\tbinom{n}{2}-\tbinom{s}{2}\right]+d\tbinom{s}{2}+\tbinom{r}{2}\\
&= 2\tbinom{n}{2}+(d-2)\tbinom{s}{2}+\tbinom{r}{2}\\
&=n(n-1)+\left[\tfrac{2n-4-r}{s}-2\right]\cdot \tbinom{s}{2}+\tbinom{r}{2}\\
&=n(n-1)+\left[\tfrac{2n-2s-4-r}{s}\right]\tbinom{s}{2}+\tbinom{r}{2}\\
&=n(n-1)+(n-s-2)(s-1)-\tfrac{r(s-1)}{2}+\tfrac{r(r-1)}{2}\\
&=n(n-1)+(n-s-2)(s-1)-\tfrac{r(s-r)}{2}\\
&\leq n(n-1)+(n-s-2)(s-1).
\end{align*}
By elementary calculus, this last expression is at most $n(n-1)+\tfrac{(n-3)^2}{4}$, with equality if and only if $s=\tfrac{n-1}{2}.$  Since $D$ is an arbitrary orientation of $G$, it follows that $\kbm(G)\leq 1+\tfrac{(n-3)^2}{4n(n-1)}<\tfrac{5}{4}.$
\end{proof}

The lower bound of Theorem~\ref{general_kappa_bar_max_bounds} is sharp if and only if $G$ is a cycle.  We believe that the upper bound of Theorem~\ref{general_kappa_bar_max_bounds} can be improved.  We will see later in Example~\ref{nine_sixteenths_example} that $\kbm(G)$ can be made arbitrarily close to $9/8$ for a minimally $2$-connected graph $G$; we know of several distinct families of minimally $2$-connected graphs which demonstrate this, but we do not know of any minimally $2$-connected graph $G$ with $\kbm(G)>9/8$.

\subsection{Bounds on the ratio \boldmath{$\kbm(G)/\kb(G)$}}
The following bounds follow from Theorem~\ref{general_kappa_bar_max_bounds} and from a bound on the average connectivity of a minimally $2$-connected graph given in~\cite{CasablancaMolOellermann2018}.

\begin{corollary}
Let $G$ be a minimally $2$-connected graph.  Then
\[
\tfrac{4}{9}< \frac{\kbm(G)}{\kb(G)}<\tfrac{5}{8}.
\]
\end{corollary}

\begin{proof}
The lower bound follows from the facts that $\kbm(G)\geq 1$ (by Theorem~\ref{Robbins}), and $\kb(G)<\tfrac{9}{4}$ (by \cite[Theorem~2.11]{CasablancaMolOellermann2018}).  The upper bound follows from the facts that $\kbm(G)<\tfrac{5}{4}$ (by Theorem~\ref{general_kappa_bar_max_bounds}), and $\kb(G)\geq 2$ since $G$ is $2$-connected.
\end{proof}

We do not know whether these bounds are sharp.  In the remainder of this section, we describe constructions of minimally $2$-connected graphs $G$ with ratio $\kbm(G)/\kb(G)$ arbitrarily close to $\tfrac{25}{54}$ (which is $\tfrac{1}{54}$ greater than the lower bound), and minimally $2$-connected graphs with ratio $\kbm(G)/\kb(G)$ arbitrarily close to $\tfrac{9}{16}$ (which is $\tfrac{1}{16}$ less than the upper bound).

We begin with two short lemmas on subdivisions of graphs.  For a graph $G$, we let $S(G)$ denote the subdivision of $G$, obtained from $G$ by subdividing every edge, i.e., replacing every edge in $G$ with a path of length $2$ joining its ends.  Formally, we define $S(G)$ to have vertex set $V(G)\cup E(G)$, and we join $u\in V(G)$ and $e\in E(G)$ by an edge whenever $u$ is an endvertex of $e$. Note if $G$ is $2$-connected, then $S(G)$ is minimally $2$-connected.


\begin{lemma}\label{KSGLemma}
For any graph $G$ of order $n$ and size $m$,
\[
K(S(G))\leq 2\left[\tbinom{n+m}{2}-\tbinom{n}{2}\right]+K(G),
\]
with equality if and only if $G$ is $2$-connected.
\end{lemma}

\begin{proof}
Let $u$ and $v$ be distinct vertices of $S(G)$.  If either $u$ or $v$ is in $E(G)$, then $\kappa_{S(G)}(u,v)\leq 2$, with equality for all such pairs if and only if $G$ is $2$-connected.  Suppose otherwise that $u,v\in V(G)$.  Any collection of $k$ internally disjoint $u$--$v$ paths in $S(G)$ corresponds to a collection of $k$ internally disjoint $u$--$v$ paths in $G$ in an obvious manner; so $\kappa_{S(G)}(u,v)=\kappa_G(u,v).$  The statement now follows by summing the connectivities between all pairs of vertices.
\end{proof}

\begin{lemma}\label{KmSGLemma}
Let $G$ be a graph of order $n$ and size $m$.  Then
\[
\Km(S(G))\leq 2\left[\tbinom{n+m}{2}-\tbinom{n}{2}\right]+\Km(G),
\]
with equality if there is an optimal orientation of $G$ that is strong.
\end{lemma}

\begin{proof}
Let $D_S$ be any orientation of $S(G)$.  Let $u$ and $v$ be distinct vertices of $S(G)$.   If either $u$ or $v$ is in $E(G)$, then $\theta_{D_S}(u,v)\leq 2$, with equality for all such pairs if $D_S$ is strong.  Suppose otherwise that $u,v\in V(G)$.  Consider the partial orientation $D$ of $G$ obtained from $D_S$ as follows.  For every edge $e\in E(G)$, say $e=uv$, orient $e$ from $u$ to $v$ if and only if both of the arcs $(u,e)$ and $(e,v)$ appear in $D_S$.  Remove all edges of $G$ that are not given an orientation in this manner (namely those edges of $G$ that are sources or sinks as vertices in $D_S$).  Then $\kappa_{D_S}(u,v)=\kappa_D(u,v)$.  Thus, we have
\[
\sum_{u,v\in V(G)}\kappa_{D_S}(u,v)=\sum_{u,v\in V(G)}\kappa_D(u,v)\leq \Km(G),
\]
with equality if and only if $D$ is an optimal orientation of $G$ (in particular, every edge of $G$ must be in $D$).  Altogether, we have
\[
\Km(S(G))\leq 2\left[\tbinom{n+m}{2}-\tbinom{n}{2}\right]+\Km(G),
\]
with equality if $D$ is a strong optimal orientation of $G$.
\end{proof}

\begin{example}\label{Ex25-54}
Let $G_k=I^k(K_4)$, with notation as in Section~\ref{CubicSection}.  The subdivision $S(G_k)$ of $G_k$ is minimally $2$-connected (since $G_k$ is $2$-connected), and we show that
\[
\lim_{k\rightarrow\infty}\frac{\kbm(S(G_k))}{\kb(S(G_k))}=\frac{25}{54}.
\]
For ease of notation, let $n=4\cdot 3^k$ (the order of $G_k$), let $m=\tfrac{3n}{2}$ (the size of $G_k$), and let $N=\binom{n+m}{2}-\binom{n}{2}.$  By a straightforward computation, we have
\begin{align}\label{AuxLimit}
\lim_{k\rightarrow\infty}\tfrac{N}{n(n-1)}=\tfrac{21}{4}.
\end{align}
By Theorem~\ref{LiftStrong}, we know that $G_k$ has a strong optimal orientation.  Hence, by Lemma~\ref{KSGLemma} and Lemma~\ref{KmSGLemma}, we have
\begin{align}\label{eq25-54}
\frac{\kbm(S(G_k))}{\kb(S(G_k))}&=\frac{\Km(S(G_k))}{2K(S(G_k))}=\frac{2N+\Km(G_k)}{2\left[2N+K(G_k)\right]}=\frac{\frac{2N}{n(n-1)}+\kbm(G_k)}{\frac{4N}{n(n-1)}+\kb(G_k)}.
\end{align}
Now using (\ref{AuxLimit}), (\ref{eq25-54}), Corollary~\ref{LiftStrong}, and the fact that $\kb(G_k)=3$, we obtain
\[
\lim_{k\rightarrow\infty}\frac{\kbm(S(G_k))}{\kb(S(G_k))}=\frac{\frac{21}{4}+1}{\frac{21}{2}+3}=\frac{25}{54}. \qedhere
\]
\end{example}

\begin{example}\label{nine_sixteenths_example}
Let $n$ be a positive integer, and define $H_{4n+1}$ as follows (see Figure~\ref{Family1}).  Let  $P:v_1v_2 \ldots v_{2n}$ be a path of order $2n$. Subdivide each edge $v_iv_{i+1}$ of $P$ for $1 \le i < 2n$, and call the new vertex $w_i$. Now add a vertex $w_{0}$ and join it to $v_1$ and $v_2$, and add a vertex $w_{2n}$ and join it to $v_{2n-1}$ and $v_{2n}$. To complete the construction of $H_{4n+1}$, add the edges $v_{i}v_{i+2}$ for $1 \le i <2n-1$.  By inspection, $H_{4n+1}$ is minimally $2$-connected, and we now show that
\[
\lim_{k\rightarrow\infty}\frac{\kbm(H_{4n+1})}{\kb(H_{4n+1}}\geq \frac{9}{16}.
\]

Apart from the pairs of vertices $v_i, v_{i+1}$ for $1 \le i <2n$, the connectivity between any pair of vertices $u,v$ in $H_{4n+1}$ is $2$. So the total connectivity of $H_{4n+1}$ is
\[K(H_{4n+1}) = 2\binom{4n+1}{2} +2n-1.\]
Hence, $\lim_{n\rightarrow\infty}\kb(H_{4n+1})=2$.

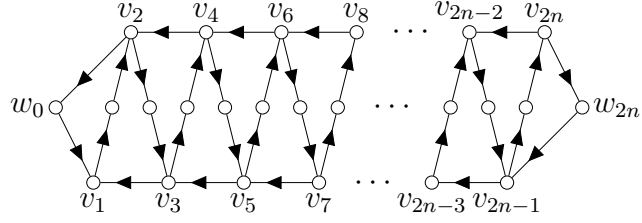
\begin{figure}[htb]
\centering{
\begin{tikzpicture}
\tikzstyle{every path}=[color =black, line width = 0.3 pt, > = triangle 45]
\pgfmathsetmacro{\n}{4}
\pgfmathtruncatemacro{\m}{\n-1}
\foreach \x in {0,...,\m}
{
\pgfmathtruncatemacro{\xl}{2*\x+1}
\pgfmathtruncatemacro{\al}{2*\x+2}
\hollowvertex (\x) at (\x,0) {};
\node[below] at (\x,0){$v_{\xl}$};
\pgfmathtruncatemacro{\a}{\n+\x}
\hollowvertex (\a) at (\x+0.5,2) {};
\node[above] at (\x+0.5,2){$v_{\al}$};
\pgfmathtruncatemacro{\b}{2*\n+\x}
\hollowvertex(\b) at (\x+0.25,1) {};
}
\pgfmathtruncatemacro{\l}{\n-2}
\foreach \x in {0,...,\l}
{
\pgfmathtruncatemacro{\c}{3*\n+\x}
\hollowvertex (\c) at (0.75+\x,1) {};
}
\foreach \x in {0,...,\m}
{
\pgfmathtruncatemacro{\a}{\n+\x}
\pgfmathtruncatemacro{\b}{2*\n+\x}
\path
(\x) edge[->-=0.75] (\b)
(\b) edge[->-=0.75] (\a)
;
}
\foreach \x in {0,...,\l}
{
\pgfmathtruncatemacro{\a}{\n+\x}
\pgfmathtruncatemacro{\b}{3*\n+\x}
\pgfmathtruncatemacro{\c}{1+\x}
\pgfmathtruncatemacro{\d}{1+\a}
\path
(\a) edge[->-=0.75] (\b)
(\b) edge[->-=0.75] (\c)
(\c) edge[->-=0.75] (\x)
(\d) edge[->-=0.75] (\a)
;
}

\node at (3.75,0) {$\cdots$};
\node at (4,1) {$\cdots$};
\node at (4.25,2) {$\cdots$};

\hollowvertex (a1) at (4.5,0) {};
\node[below] at (4.5,0) {$v_{2n-3}$};
\hollowvertex (a2) at (4.75,1) {};
\hollowvertex (a3) at (5,2) {};
\node[above] at (5,2) {$v_{2n-2}$};
\hollowvertex (a4) at (5.25,1) {};
\hollowvertex (a5) at (5.5,0) {};
\node[below] at (5.5,0) {$v_{2n-1}$};
\hollowvertex (a6) at (5.75,1) {};
\hollowvertex (a7) at (6,2) {};
\node[above] at (6,2) {$v_{2n}$};
\path
(a1) edge[->-=0.75] (a2)
(a2) edge[->-=0.75] (a3)
(a3) edge[->-=0.75] (a4)
(a4) edge[->-=0.75] (a5)
(a5) edge[->-=0.75] (a6)
(a6) edge[->-=0.75](a7)
(a5) edge[->-=0.75] (a1)
(a7) edge[->-=0.75] (a3);

\hollowvertex (first) at (-0.5,1) {};
\node[left] at (-0.5,1){$w_0$};
\hollowvertex (last) at (6.5,1) {};
\node[right] at (6.5,1){$w_{2n}$};
\path
(\n) edge[->-=0.75] (first)
(first) edge[->-=0.75] (0)
(a7) edge[->-=0.75] (last)
(last) edge[->-=0.75] (a5);
\end{tikzpicture}
}
\caption{A family of minimally $2$-connected graphs with ratio tending to $\tfrac{9}{16}$}\label{Family1}
\end{figure}

We now describe an orientation $D_{4n+1}$ of $H_{4n+1}$ (see Figure~\ref{Family1}). For every $1 \le i < 2n$, orient the path $v_iw_{i}v_{i+1}$ from $v_i$ to $v_{i+1}$. Orient the path $v_2w_{0}v_1$ from $v_2$ to $v_1$, and orient the path $v_{2n}w_{2n}v_{2n-1}$ from $v_{2n}$ to $v_{2n-1}$. Orient the path $v_{2n}v_{2n-2} \ldots v_2$ from $v_{2n}$ to $v_2$, and the path $v_{2n-1}v_{2n-3} \ldots v_1$ from $v_{2n-1}$ to $v_1$.  One can verify that $D_{4n+1}$ is strong.  Note that $H_{4n+1}$ has $2n$ vertices of degree at least $3$ (the $v_i$'s) and $2n+1$ vertices of degree $2$ (the $w_i$'s).  Note that for any pair $v_i$, $v_j$ with $i<j$, we have $\kappa_{D_{4n+1}}(v_i,v_j)=1$ and $\kappa_{D_{4n+1}}(v_j,v_i)=2$, hence $\theta_{D_{4n+1}}(v_i,v_j)=3$. Therefore, the total connectivity of $D_{4n+3}$ is given by
\[K(D_{4n+1})=2\binom{4n+1}{2}+ \binom{2n}{2}.\]
So $\lim_{n\rightarrow\infty}\kb(D_{4n+1})=\tfrac{9}{8}$, and we certainly have $\kbm(H_{4n+1})\geq \kb(D_{4n+1})$ for all $n$.  It follows that
\[
\lim_{n\rightarrow\infty}\frac{\kbm(H_{4n+1})}{\kb(H_{4n+1})}\geq\lim_{n\rightarrow\infty}\frac{\kb(D_{4n+1})}{\kb(H_{4n+1})}=\frac{9}{16}. \qedhere
\]
\end{example}

\section{Maximal outerplanar graphs and 2-trees}\label{MOPs}

A graph is {\em outerplanar} if it can be drawn in the plane so that no two of its edges cross and all of its vertices appear on the boundary of the outer face.
A graph is {\em maximal outerplanar} if it is outerplanar and the addition of any edge destroys this property.  It is known that every maximal outerplanar graph is a \emph{$2$-tree}.  For $k \ge 1$, the $k$-{\em trees} are defined recursively as follows: the complete graph $K_k$ is a $k$-tree, and if $T$ is a $k$-tree, then the graph obtained from $T$ by adding a new vertex and joining it to every vertex in a $k$-clique of $T$ is a $k$-tree.  Note that trees are precisely the $1$-trees.

It was shown in \cite{DankelmannOellermann2003} that all maximal outerplanar graphs of the same order have the same average connectivity.

\begin{theorem} {\em \cite{DankelmannOellermann2003}}
If $G$ is a maximal outerplanar graph of order $n$, then \[\kb(G)=2+\frac{2n-6}{n(n-1)}.\]
\end{theorem}

 So, for  maximal outerplanar graphs of a fixed order, the ratio $\kbm(G)/\kb(G)$ is maximized or minimized exactly when $\kbm(G)$ is maximized or minimized, respectively.  So it suffices to focus on bounds for $\kbm(G)$ if $G$ is maximal outerplanar. We show that if $G$ is a maximal outerplanar graph, then $\kbm(G) \le \frac{3}{2}+o(1)$. Moreover, this bound is asymptotically sharp.  We conjecture that $\kbm(G)\geq 19/18$ for every maximal outerplanar graph of order at least $4$, and we demonstrate that if this conjecture is true, then the bound is asymptotically sharp.

It is known that every maximal outerplanar graph is a $2$-tree. As was the case for maximal outerplanar graphs, one can show (by induction), that if $G$ is a $2$-tree of order $n \ge 3$, then $\kb(G)=2+\frac{2n-6}{n(n-1)}$. Thus, for  $2$-trees of a fixed order, the ratio $\kbm(G)/\kb(G)$ is maximized or minimized exactly when $\kbm(G)$ is maximized or minimized, respectively. We show that if $G$ is a $2$-tree, then the bound $\kbm(G) \ge 1$, guaranteed by Robbins' Theorem, is asymptotically sharp.

\subsection{Maximal outerplanar graphs}\label{MOPsection}

We use the following notation throughout this section.
Let $G$ be a maximal outerplanar graph with a given embedding
in the plane. We say that an edge is an {\em outer edge} of $G$ if
it is part of the cycle $C$ which forms the boundary of the outer face
of $G$. An edge that is not an outer edge is a {\em chord} of $G$.
We denote by $G'$ the spanning subgraph of $G$ whose edges
are the chords of $G$.

The {\em weak dual} $G^*$ of $G$ is the graph
whose vertices are the faces of $G$ distinct from the outer
face, and two vertices of $G^*$ are adjacent if, as faces of $G$, their boundaries share an edge.
If $u^*$ is a vertex of $G^*$, then $V(u^*)$ denotes the set of
vertices of $G$ that are on the boundary of $u^*$.  It is well-known that
$G^*$ is a tree of maximum degree at most 3, and it is easy to see that
a vertex of $G^*$ has degree 3 if and only if its boundary consists
of three chords.

\begin{lemma} \label{la:y>=x/2+2}
Let $G$ be a maximal outerplanar graph of order $n$.
Let $A$ be the set of vertices of $G$ that are on a $4$-cycle in $G'$.
Let $B_2$ be the set of vertices of degree 2 in $G$. Then
\[ |B_2| \geq \frac{1}{2}|A| + 2. \]
\end{lemma}
\begin{proof}
Denote the set of vertices of  $G^*$ whose degree is 3 and who are adjacent to
some other vertex of degree 3 by $V_3^*$. We prove the lemma by bounding
$|V_3^*|$ from above in terms of $|B_2|$, and from below in terms of $|A|$.

We first bound  $|V_3^*|$ from above in terms of $|B_2|$.
Since a vertex of $G$ has degree 2 if and only if its incident edges are both on
the outer cycle, there is a natural bijection between the vertices of $G$ of degree 2
and the leaves of $G^*$. Hence $|B_2|=n_1^*$, where $n_i^*$ is the number of
vertices of $G^*$ of degree $i$. On the other hand, since $G^*$ is a tree of maximum
degree at most 3, we have $n^*_3 = n^*_1-2$. Hence
\begin{align}  \label{eq:v3-bounded-by-b2}
|V_3^*| \leq n^*_3 = n_1^* - 2 = |B_2| -2.
\end{align}

We now bound $|V_3^*|$ from below in terms of $|A|$.
We first show that
\begin{align} \label{eq:A-is-union-of-V(u)}
A = \bigcup_{u^* \in V_3^*} V(u^*).
\end{align}

To see that $A \subseteq \bigcup_{u^* \in V_3^*} V(u^*)$, let $v\in A$.
Then $v$ is on some $4$-cycle $C_v$ whose edges are chords.
Since $G$ is maximal outerplanar, $C_v$ also has a chord $e$. The two
faces $u_1^*$ and $u_2^*$ of $G$ that have $e$ on their respective boundaries
are adjacent in $G^*$.  Further, since the edges on their boundaries are all chords, we see that $u_1^*$ and $u_2^*$ have degree
3 in $G^*$.  Hence, we have $u^*_1, u^*_2 \in V_3^*$ and $v \in V(u_1^*) \cup V(u_2^*)$,
and it follows that $v\in \bigcup_{u^* \in V_3^*} V(u^*)$.

To see that $\bigcup_{u^* \in V_3^*} V(u^*)\subseteq A$, let
$v \in \bigcup_{u^* \in V_3^*} V(u^*)$. Then $v \in V(u_1^*)$ for some
$u_1^* \in V_3^*$. By the definition of $V_3^*$, the vertex $u_1^*$ has a neighbour
$u_2^*$ in $G^*$ of degree 3. Since $V(u_1^*)$ and $V(u^*_2)$ share
an edge of $G$, and since the edges on the boundary of $u_i^*$, for $i \in \{1,2\}$,
are all chords, there is a $4$-cycle of $G$ containing the vertices
of $V(u_1^*) \cup V(u_2^*)$, whose edges are chords. Hence $v\in A$.
This proves \eqref{eq:A-is-union-of-V(u)}.

Let $H_1^*, H_2^*,\ldots, H_k^*$ be the components of the graph $G^*[V_3^*]$.
Then each $H_i^*$ is a tree on at least two vertices.
Clearly, if $H_i^*$ has two vertices, then $\left|\bigcup_{u^*\in V(H_i^*)} V(u^*)\right| = 4$.
Every additional vertex in $H_i^*$ increases $\left|\bigcup_{u^*\in V(H_i^*)} V(u^*)\right|$ by
one, hence
\[
\left|\textstyle\bigcup_{u^*\in V(H_i^*)} V(u^*)\right| = \left|V(H_i^*)\right|+2 \leq 2 \left|V(H_i^*)\right|.
\]
Summation over $i=1,2,\ldots,k$ yields
\begin{align}
|A| & =  \left|\textstyle\bigcup_{u^*\in V_3^*} V(u^*)\right| \nonumber \\
    & =  \left|\textstyle\bigcup_{i=1}^k \bigcup_{u^*\in V(H_i)} V(u^*)\right| \nonumber \\
    & \leq   \sum_{i=1}^k \left|\textstyle\bigcup_{u^*\in V(H_i)} V(u^*)\right| \nonumber \\
    & \leq  \sum_{i=1}^k 2\left|V(H_i^*)\right| \nonumber \\
    & =  2|V_3^*|.   \label{eq:v3-bounded-by-a}
\end{align}
Combining \eqref{eq:v3-bounded-by-b2} with \eqref{eq:v3-bounded-by-a}
now yields the statement of the lemma.
\end{proof}

\begin{theorem}  \label{theo:upper-bound-for-MOP}
If $G$ is a maximal outerplanar graph of order $n$, then
\[ \bar{\kappa}_{max}(G) \leq \frac{3}{2} + \frac{n-5}{n(n-1)}. \]
\end{theorem}
\begin{proof}
Let $D$ be an arbitrary orientation of $G$. Let $C$, $A$ and $B_2$ be as defined
above, and let $u,v \in V(G)$. We bound $\theta(u,v)$ from above.

First assume that $uv \in E(G)$. We prove that
\begin{align} \label{eq:bound-on-theta-adjacent}
\theta(u,v) \leq \begin{cases}
     5, & \textrm{if $uv$ is a chord;} \\
     3, & \textrm{if $uv$ is an outer edge.}
        \end{cases}
\end{align}
If $uv$ is a chord, then $u$ and $v$ have two common neighbours, say $a$ and $b$.
Every path between $u$ and $v$ in $G$, and thus in $D$, contains either $a$
or $b$ or the edge $uv$. Hence, apart from the path consisting of the edge $uv$, there exist
at most two internally disjoint directed $u$--$v$ paths in $D$, and at most two internally disjoint
directed $v$--$u$ paths in $D$. Hence,  if $uv$ is a chord, then $\theta(u,v)\leq 5$.

If $uv$ is an outer edge, then $u$ and $v$ have exactly one common neighbour, say $a$.
Every path between $u$ and $v$ in $G$ contains either $a$
or the edge $uv$.  Hence, apart from the path consisting of the edge $uv$, there exist no two internally disjoint directed $u$--$v$ paths in $D$, and no two internally disjoint directed $v$--$u$ paths in $D$. Hence, if $uv$ is an outer edge, then $\theta(u,v)\leq 3$.

Now assume that $uv \notin E(G)$. We prove that
\begin{align} \label{eq:bound-on-theta-nonadjacent}
\theta(u,v) \leq \begin{cases}
     2, & \textrm{if $\{u,v\} \cap B_2 \neq \emptyset$;} \\
     4, & \textrm{if $\{u,v\} \subseteq A$;} \\
     3, & \textrm{otherwise.}
        \end{cases}
\end{align}
If $\{u,v\} \cap B_2 \neq \emptyset$, then
$\theta(u,v) \leq \min\{{\rm deg}_G(u), {\rm deg}_G(v)\}=2$.
So assume that $\{u,v\} \cap B_2 = \emptyset$.
Clearly, since $u$ and $v$ are nonadjacent, and since $G$ is maximal outerplanar,
there exist two adjacent vertices $a$ and $b$ of $G$ that separate $u$ and $v$.
Hence, there exist at most two internally disjoint directed $u$--$v$ paths in $D$, and at most two internally disjoint directed $v$--$u$ paths in $D$.
It follows that $\theta(u,v) \leq 4$.  In order to complete the proof of
\eqref{eq:bound-on-theta-nonadjacent}, it suffices to show the following:
\begin{align}  \label{eq:u-is-on-a-4-cycle}
\textrm{If $\theta(u,v)=4$, then $u$ is on a $4$-cycle in $G'$. }
\end{align}
Assume that $\theta(u,v)=4$.
We may assume that if $C$ is traversed in clockwise direction, then
$u$, $a$ and $b$ appear in this order.
Let $u_1, u_2,\ldots,u_k$ be the neighbours of $u$ in clockwise order,
where $u_1$ and $u_k$ are the neighbours of $u$ in $C$.
Since $G$ is outerplanar, there exists $j$ such that $u_1, \ldots, u_j$
are in the  $u$--$a$ subpath, and $u_{j+1},\ldots, u_k$ are on the $b$--$u$ subpath
of $C$. Then $\{u_j, u_{j+1}\}$ separates $u$ and $v$ in $G$.
Also, $uu_j$ is a chord of $G$, since otherwise, if $uu_j$ was an outer edge,
then every $u$--$v$ path in $G$ passes either through $uu_j$ or through $u_{j+1}$,
implying that $\theta(u,v) \leq 3$. Similarly, $uu_{j+1}$ is a chord.

There exists a common neighbour $c$ of $u_j$ and $u_{j+1}$ distinct from $u$.
We show that $u_jc$ is a chord of $G$. Suppose to the contrary that
$u_jc$ is an outer edge of $G$. Since every $u$--$v$ path in $G$ passes through
$\{u_j, u_{j+1}\}$, it follows that every $u$--$v$ path in $G$ passes either
through the edge $u_jc$ or the vertex $u_{j+1}$, which implies that
$\theta(u,v) \leq 3$, a contradiction. Hence $u_jc$ is a chord. Similarly
we show that $u_{j+1}c$ is a chord.
We conclude that $u, u_j, c, u_{j+1}, u$ is
a $4$-cycle whose edges are chords, so $u\in A$, and
\eqref{eq:u-is-on-a-4-cycle} follows.

We use \eqref{eq:bound-on-theta-adjacent} and \eqref{eq:bound-on-theta-nonadjacent}
to bound the total connectivity of $D$. Let $x=|A|$ and $y=|B_2|$.
First note that $G$ has $n$ unordered pairs $\{u,v\}$ of vertices that are joined by an outer edge.
Of these, exactly $2y$ pairs involve a vertex of degree 2, so that $\theta(u,v) \leq 2$
in this case, and the remaining $n-2y$ pairs satisfy $\theta(u,v) \leq 3$.
Next, note that $G$ has $n-3$ pairs $\{u,v\}$ of vertices that are joined by a chord,
and for these we have $\theta(u,v) \leq 5$.

Of the $\binom{n}{2}-2n+3$ pairs $\{u,v\}$ of nonadjacent vertices, at most
$\binom{x}{2}$ are contained in $A$, so $\theta(u,v) \leq 4$ for these pairs.
There are $\binom{n}{2} - \binom{n-y}{2}$ unordered pairs $\{u,v\}$
of vertices involving a vertex of degree 2, and $2y$ of these are joined by
an outer edge, while none of them are joined by chords. Hence there are
$\binom{n}{2} - \binom{n-y}{2} -2y$ pairs of nonadjacent vertices involving
a vertex of degree 2, so that $\theta(u,v)\leq 2$. The remaining
$\binom{n-y}{2} - \binom{x}{2} -2n+3 + 2y$ pairs satisfy $\theta(u,v) \leq 3$.
Summation of $\theta(u,v)$ over all unordered pairs $\{u,v\}$ yields that
\[
\sum_{ \{u,v\} \subseteq V(G)} \theta(u,v)
    \leq  2 \binom{n}{2} + \binom{n-y}{2} + \binom{x}{2} + 2n-6.
\]
Now $y\geq \frac{1}{2}x + 2$ by Lemma \ref{la:y>=x/2+2}, hence
\begin{align}
\sum_{ \{u,v\} \subseteq V(G)} \theta(u,v)
   & \leq   2 \binom{n}{2} + \binom{n-2 - \frac{1}{2}x}{2} + \binom{x}{2} + 2n-6 \nonumber \\
   & =  \frac{3}{2}n^2 - \frac{1}{2}n - 5 - \frac{1}{2}nx + \frac{5}{8}x^2 - \frac{5}{4}x.
      \label{eq:mops-bound-in-terms-of-x-y}
\end{align}
Since $y \geq \frac{1}{2}x+2$, we have $n \geq x+y \geq \frac{3}{2}x+2$, and thus
$x \leq \frac{2}{3}n - \frac{4}{3}$. Elementary calculus shows that the right hand side
of \eqref{eq:mops-bound-in-terms-of-x-y}, as a function of $x$, is maximized
subject to $0 \leq x \leq \frac{2}{3}n - \frac{4}{3}$ when $x=0$.
Substituting this yields
\[  \sum_{ \{u,v\} \subseteq V(G)} \theta(u,v)
     \leq  \frac{3}{2}n^2 - \frac{1}{2}n - 5, \]
and dividing by $n(n-1)$ yields the theorem.
\end{proof}

The bound of Theorem \ref{theo:upper-bound-for-MOP} is asymptotically sharp.
Let $G_{2n}$ be the maximal outerplanar graph
obtained from the path $P_{2n}:v_1v_2\dots v_{2n}$ by adding the edges of the paths $Q:v_1v_3\dots v_{2n-1}$ and $R:v_2v_4\dots v_{2n}$. Note that $G_{2n}$ is the square of the path of order $2n$.
Let $D_{2n}$ be the orientation of $G_{2n}$ obtained by directing the edges
of $P_{2n}$ from $v_1$ to $v_{2n}$, the edges of $Q$ from $v_{2n-1}$ to $v_1$,
and the edges of $R$ from $v_{2n}$ to $v_2$ (see Figure~\ref{Snake}).  It is easy to see that all unordered
pairs $\{u,v\}$ of vertices of $D_{2n}$ satisfy $\theta(u,v)=3$, except for those
$4n-3$ pairs that involve a vertex of degree 2 in $G$,
for which we have $\theta(u,v)=2$. Hence, we have $\kbm(G_{2n})\geq\kb(D_{2n})= \frac{3}{2} - \frac{4n-3}{2n(2n-1)}$.  

\begin{figure}
\centering{
\begin{tikzpicture}
\tikzstyle{every path}=[color =black, line width = 0.3 pt, > = triangle 45]
\pgfmathsetmacro{\n}{4}
\pgfmathtruncatemacro{\m}{\n-1}
\foreach \x in {0,...,\m}
{
\pgfmathtruncatemacro{\xl}{2*\x+1}
\pgfmathtruncatemacro{\al}{2*\x+2}
\hollowvertex (\x) at (\x,0) {};
\node[below] at (\x,0){$v_{\xl}$};
\pgfmathtruncatemacro{\a}{\n+\x}
\hollowvertex (\a) at (\x+0.5,1) {};
\node[above] at (\x+0.5,1){$v_{\al}$};
}
\pgfmathtruncatemacro{\l}{\n-2}
\foreach \x in {0,...,\m}
{
\pgfmathtruncatemacro{\a}{\n+\x}
\pgfmathtruncatemacro{\b}{2*\n+\x}
\path
(\x) edge[->-=0.85] (\a)
;
}
\foreach \x in {0,...,\l}
{
\pgfmathtruncatemacro{\a}{\n+\x}
\pgfmathtruncatemacro{\c}{1+\x}
\pgfmathtruncatemacro{\d}{1+\a}
\path
(\a) edge[->-=0.85] (\c)
(\c) edge[->-=0.85] (\x)
(\d) edge[->-=0.85] (\a)
;
}

\node at (3.75,0) {$\cdots$};
\node at (4.25,1) {$\cdots$};

\hollowvertex (a1) at (4.5,0) {};
\node[below] at (4.5,0) {$v_{2n-3}$};
\hollowvertex (a3) at (5,1) {};
\node[above] at (5,1) {$v_{2n-2}$};
\hollowvertex (a5) at (5.5,0) {};
\node[below] at (5.5,0) {$v_{2n-1}$};
\hollowvertex (a7) at (6,1) {};
\node[above] at (6,1) {$v_{2n}$};
\path
(a1) edge[->-=0.85] (a3)
(a3) edge[->-=0.85] (a5)
(a5) edge[->-=0.85] (a7)
(a5) edge[->-=0.85] (a1)
(a7) edge[->-=0.85] (a3);
\end{tikzpicture}
}
\caption{The orientation $D_{2n}$ of maximal outerplanar graph $G_{2n}$}\label{Snake}
\end{figure}
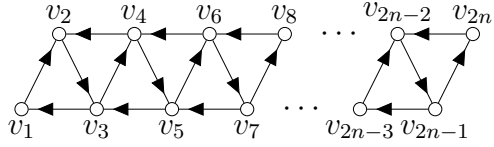

\medskip

We now discuss a lower bound for $\KBM(G)$ if $G$ is a maximal outerplanar graph. Since maximal outerplanar graphs of order at least $3$ are $2$-connected, it follows immediately, from Robbins' Theorem, that $\kbm(G)\geq 1$ for every maximal outerplanar graph $G$ of order at least $3$. Moreover, this bound is tight since $\kbm(K_3)=1$. However, the graph $K_3$ seems exceptional.

Table~\ref{MOPTable} gives the minimum value of $\kbm(G)$ taken over all maximal outerplanar graphs $G$ of order $n$ for $3\leq n\leq 9$.  For $3\leq n\leq 8$, the fan $F_n$ (the join of $K_1$ and $P_{n-1}$)  is the unique maximal outerplanar graph of order $n$ that realizes this minimum value.  For $n=9$, the fan $F_9$ and one other graph attain the minimum value of $\kbm$.  We do not pursue the details, but it appears that $\lim_{n\rightarrow\infty}\kbm(F_n)=\tfrac{5}{4}.$

\begin{table}
\centering{
\begin{tabular}{c c}
Order & Minimum value of $\kbm$\\\hline
3 & $1$\\
4 & $13/12$\\
5 & $23/20$\\
6 & $7/6$\\
7 & $25/21$\\
8 & $67/56$\\
9 & $29/24$
\end{tabular}
}
\caption{The minimum value of $\kbm$ over all maximal outerplanar graphs of a given order}
\label{MOPTable}
\end{table}

Based on the information in Table~\ref{MOPTable}, one might initially guess that the minimum value of $\kbm$ must increase with the order. However, we now describe an infinite family of maximal outerplanar graphs for which $\KBM$ is asymptotically at most $\frac{19}{18}+o(1)$.  We conjecture that $\kbm(G)\geq 19/18$ for every maximal outerplanar graph $G$ of order at least $4$. The following example was inspired by the the example from \cite{HenningOellermann2004} that demonstrates sharpness for the lower bound of Theorem \ref{henningoellermann}.

\begin{example}
Define a {\em trigon} as a triangle with every edge coloured red.  Define a {\em lozenge} as a $K_4-e$ in which the edges of a perfect matching are coloured red, and all other edges are coloured black. Construct graph $G_0$ as follows: start with a trigon, and glue a lozenge red-on-red to every edge of the trigon, so that each vertex of the trigon has degree $5$ in the resulting graph. See Figure~\ref{G0} for illustrations of a trigon, a lozenge, and the graph $G_0$.

\begin{figure}[htb]
\centering
\begin{subfigure}[t]{0.3\textwidth}
\centering
\begin{tikzpicture}[scale=0.8]
\hollowvertex(u) at (-1,-1.73/3) {};
\hollowvertex(v) at (1,-1.73/3) {};
\hollowvertex(w) at (0,2*1.73/3) {};
\path[line width=3pt, color=red]
(u) edge (v)
(v) edge (w)
(w) edge (u)
;
\end{tikzpicture}
\caption{A trigon}
\end{subfigure}%
\begin{subfigure}[t]{0.3\textwidth}
\centering
\begin{tikzpicture}[scale=0.8]
\hollowvertex(u0) at (-1,-1.73/3) {};
\hollowvertex(v0) at (1,-1.73/3) {};
\hollowvertex(u) at (1,-2-1.73/3) {};
\hollowvertex (u1) at (-1,-2-1.73/3) {};
\path[line width=3pt, color=red]
(v0) edge (u0)
(u1) edge (u)
;
\path
(u) edge (u0)
(v0) edge (u)
(u0) edge (u1)
;
\end{tikzpicture}
\caption{A lozenge}
\end{subfigure}%
\begin{subfigure}[t]{0.4\textwidth}
\centering
\begin{tikzpicture}[scale=0.8]
\hollowvertex(u0) at (-1,-1.73/3) {};
\hollowvertex(v0) at (1,-1.73/3) {};
\hollowvertex(u) at (1,-2-1.73/3) {};
\hollowvertex (u1) at (-1,-2-1.73/3) {};
\path[line width=3pt, color=red]
(v0) edge (u0)
(u1) edge (u)
;
\path
(u) edge (u0)
(v0) edge (u)
(u0) edge (u1)
;
\begin{scope}[rotate=120]
\hollowvertex(u0) at (-1,-1.73/3) {};
\hollowvertex(v0) at (1,-1.73/3) {};
\hollowvertex(u) at (1,-2-1.73/3) {};
\hollowvertex (u1) at (-1,-2-1.73/3) {};
\path[line width=3pt, color=red]
(v0) edge (u0)
(u1) edge (u)
;
\path
(u) edge (u0)
(v0) edge (u)
(u0) edge (u1)
;
\end{scope}
\begin{scope}[rotate=240]
\hollowvertex(u0) at (-1,-1.73/3) {};
\hollowvertex(v0) at (1,-1.73/3) {};
\hollowvertex(u) at (1,-2-1.73/3) {};
\hollowvertex (u1) at (-1,-2-1.73/3) {};
\path[line width=3pt, color=red]
(v0) edge (u0)
(u1) edge (u)
;
\path
(u) edge (u0)
(v0) edge (u)
(u0) edge (u1)
;
\end{scope}
\end{tikzpicture}
\caption{The graph $G_0$}
\end{subfigure}
\caption{A trigon, a lozenge, and the graph $G_0$} \label{G0}
\end{figure}
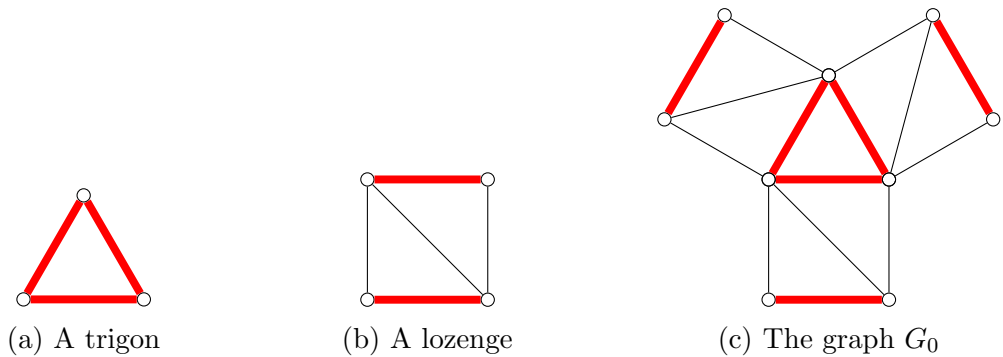

Suppose $G_{i-1}$ has been constructed for some $i >0$. Construct $G_i$ from $G_{i-1}$ as follows: glue a trigon to every red outer edge of $G_{i-1}$, and then glue two lozenges (red-on-red) onto the two red outer edges of each trigon, so that each vertex in the trigon has degree $5$ in the resulting graph (see Figure~\ref{G1}).

\begin{figure}
\begin{subfigure}[t]{0.49\textwidth}
\centering
\begin{tikzpicture}[scale=0.51]
\hollowvertex(u0) at (-1,-1.73/3) {};
\hollowvertex(v0) at (1,-1.73/3) {};
\hollowvertex(u) at (1,-2-1.73/3) {};
\hollowvertex (u1) at (-1,-2-1.73/3) {};
\path[line width=3pt, color=red]
(v0) edge (u0)
(u1) edge (u)
;
\path
(u) edge (u0)
(v0) edge (u)
(u0) edge (u1)
;
\begin{scope}[shift={(0,-3.16)},rotate=60]
\hollowvertex(u0) at (-1,-1.73/3) {};
\hollowvertex(v0) at (1,-1.73/3) {};
\hollowvertex(u) at (1,-2-1.73/3) {};
\hollowvertex (u1) at (-1,-2-1.73/3) {};
\path[line width=3pt, color=red]
(v0) edge (u0)
(u1) edge (u)
;
\path
(u) edge (u0)
(v0) edge (u)
(u0) edge (u1)
;
\draw [dashed,opacity=0] (u1) to [bend right=90,looseness=2.5]  (u);
\end{scope}
\begin{scope}[shift={(0,-3.16)},rotate=-60]
\hollowvertex(u0) at (-1,-1.73/3) {};
\hollowvertex(v0) at (1,-1.73/3) {};
\hollowvertex(u) at (1,-2-1.73/3) {};
\hollowvertex (u1) at (-1,-2-1.73/3) {};
\path[line width=3pt, color=red]
(v0) edge (u0)
(u1) edge (u)
;
\path
(u) edge (u0)
(v0) edge (u)
(u0) edge (u1)
;
\draw [dashed,opacity=0] (u1) to [bend right=90,looseness=2.5]  (u);
\end{scope}
\begin{scope}[rotate=120]
\hollowvertex(u0) at (-1,-1.73/3) {};
\hollowvertex(v0) at (1,-1.73/3) {};
\hollowvertex(u) at (1,-2-1.73/3) {};
\hollowvertex (u1) at (-1,-2-1.73/3) {};
\path[line width=3pt, color=red]
(v0) edge (u0)
(u1) edge (u)
;
\path
(u) edge (u0)
(v0) edge (u)
(u0) edge (u1)
;
\begin{scope}[shift={(0,-3.16)},rotate=60]
\hollowvertex(u0) at (-1,-1.73/3) {};
\hollowvertex(v0) at (1,-1.73/3) {};
\hollowvertex(u) at (1,-2-1.73/3) {};
\hollowvertex (u1) at (-1,-2-1.73/3) {};
\path[line width=3pt, color=red]
(v0) edge (u0)
(u1) edge (u)
;
\path
(u) edge (u0)
(v0) edge (u)
(u0) edge (u1)
;
\draw [dashed,opacity=0] (u1) to [bend right=90,looseness=2.5]  (u);
\end{scope}
\begin{scope}[shift={(0,-3.16)},rotate=-60]
\hollowvertex(u0) at (-1,-1.73/3) {};
\hollowvertex(v0) at (1,-1.73/3) {};
\hollowvertex(u) at (1,-2-1.73/3) {};
\hollowvertex (u1) at (-1,-2-1.73/3) {};
\path[line width=3pt, color=red]
(v0) edge (u0)
(u1) edge (u)
;
\path
(u) edge (u0)
(v0) edge (u)
(u0) edge (u1)
;
\draw [dashed,opacity=0] (u1) to [bend right=90,looseness=2.5]  (u);
\end{scope}
\end{scope}
\begin{scope}[rotate=240]
\hollowvertex(u0) at (-1,-1.73/3) {};
\hollowvertex(v0) at (1,-1.73/3) {};
\hollowvertex(u) at (1,-2-1.73/3) {};
\hollowvertex (u1) at (-1,-2-1.73/3) {};
\path[line width=3pt, color=red]
(v0) edge (u0)
(u1) edge (u)
;
\path
(u) edge (u0)
(v0) edge (u)
(u0) edge (u1)
;
\begin{scope}[shift={(0,-3.16)},rotate=60]
\hollowvertex(u0) at (-1,-1.73/3) {};
\hollowvertex(v0) at (1,-1.73/3) {};
\hollowvertex(u) at (1,-2-1.73/3) {};
\hollowvertex (u1) at (-1,-2-1.73/3) {};
\path[line width=3pt, color=red]
(v0) edge (u0)
(u1) edge (u)
;
\path
(u) edge (u0)
(v0) edge (u)
(u0) edge (u1)
;
\draw [dashed,opacity=0] (u1) to [bend right=90,looseness=2.5]  (u);
\end{scope}
\begin{scope}[shift={(0,-3.16)},rotate=-60]
\hollowvertex(u0) at (-1,-1.73/3) {};
\hollowvertex(v0) at (1,-1.73/3) {};
\hollowvertex(u) at (1,-2-1.73/3) {};
\hollowvertex (u1) at (-1,-2-1.73/3) {};
\path[line width=3pt, color=red]
(v0) edge (u0)
(u1) edge (u)
;
\path
(u) edge (u0)
(v0) edge (u)
(u0) edge (u1)
;
\draw [dashed,opacity=0] (u1) to [bend right=90,looseness=2.5]  (u);
\end{scope}
\end{scope}
\end{tikzpicture}
\caption{The graph $G_1$}
\label{G1}
\end{subfigure}
\begin{subfigure}[t]{0.49\textwidth}
\centering
\begin{tikzpicture}[scale=0.51]
\hollowvertex(u0) at (-1,-1.73/3) {};
\hollowvertex(v0) at (1,-1.73/3) {};
\hollowvertex(u) at (1,-2-1.73/3) {};
\hollowvertex (u1) at (-1,-2-1.73/3) {};
\path[line width=3pt, color=red]
(v0) edge (u0)
(u1) edge (u)
;
\path
(u) edge (u0)
(v0) edge (u)
(u0) edge (u1)
;
\begin{scope}[shift={(0,-3.16)},rotate=60]
\hollowvertex(u0) at (-1,-1.73/3) {};
\hollowvertex(v0) at (1,-1.73/3) {};
\hollowvertex(u) at (1,-2-1.73/3) {};
\hollowvertex (u1) at (-1,-2-1.73/3) {};
\path[line width=3pt, color=red]
(v0) edge (u0)
(u1) edge (u)
;
\path
(u) edge (u0)
(v0) edge (u)
(u0) edge (u1)
;
\draw [dashed] (u1) to [bend right=90,looseness=2.5]  (u);
\node at (0,-3.25) {$M_1$};
\end{scope}
\begin{scope}[shift={(0,-3.16)},rotate=-60]
\hollowvertex(u0) at (-1,-1.73/3) {};
\hollowvertex(v0) at (1,-1.73/3) {};
\hollowvertex(u) at (1,-2-1.73/3) {};
\hollowvertex (u1) at (-1,-2-1.73/3) {};
\path[line width=3pt, color=red]
(v0) edge (u0)
(u1) edge (u)
;
\path
(u) edge (u0)
(v0) edge (u)
(u0) edge (u1)
;
\draw [dashed] (u1) to [bend right=90,looseness=2.5]  (u);
\node at (0,-3.25) {$M_1$};
\end{scope}
\begin{scope}[rotate=120]
\hollowvertex(u0) at (-1,-1.73/3) {};
\hollowvertex(v0) at (1,-1.73/3) {};
\hollowvertex(u) at (1,-2-1.73/3) {};
\hollowvertex (u1) at (-1,-2-1.73/3) {};
\path[line width=3pt, color=red]
(v0) edge (u0)
(u1) edge (u)
;
\path
(u) edge (u0)
(v0) edge (u)
(u0) edge (u1)
;
\begin{scope}[shift={(0,-3.16)},rotate=60]
\hollowvertex(u0) at (-1,-1.73/3) {};
\hollowvertex(v0) at (1,-1.73/3) {};
\hollowvertex(u) at (1,-2-1.73/3) {};
\hollowvertex (u1) at (-1,-2-1.73/3) {};
\path[line width=3pt, color=red]
(v0) edge (u0)
(u1) edge (u)
;
\path
(u) edge (u0)
(v0) edge (u)
(u0) edge (u1)
;
\draw [dashed] (u1) to [bend right=90,looseness=2.5]  (u);
\node at (0,-3.25) {$M_1$};
\end{scope}
\begin{scope}[shift={(0,-3.16)},rotate=-60]
\hollowvertex(u0) at (-1,-1.73/3) {};
\hollowvertex(v0) at (1,-1.73/3) {};
\hollowvertex(u) at (1,-2-1.73/3) {};
\hollowvertex (u1) at (-1,-2-1.73/3) {};
\path[line width=3pt, color=red]
(v0) edge (u0)
(u1) edge (u)
;
\path
(u) edge (u0)
(v0) edge (u)
(u0) edge (u1)
;
\draw [dashed] (u1) to [bend right=90,looseness=2.5]  (u);
\node at (0,-3.25) {$M_1$};
\end{scope}
\end{scope}
\begin{scope}[rotate=240]
\hollowvertex(u0) at (-1,-1.73/3) {};
\hollowvertex(v0) at (1,-1.73/3) {};
\hollowvertex(u) at (1,-2-1.73/3) {};
\hollowvertex (u1) at (-1,-2-1.73/3) {};
\path[line width=3pt, color=red]
(v0) edge (u0)
(u1) edge (u)
;
\path
(u) edge (u0)
(v0) edge (u)
(u0) edge (u1)
;
\begin{scope}[shift={(0,-3.16)},rotate=60]
\hollowvertex(u0) at (-1,-1.73/3) {};
\hollowvertex(v0) at (1,-1.73/3) {};
\hollowvertex(u) at (1,-2-1.73/3) {};
\hollowvertex (u1) at (-1,-2-1.73/3) {};
\path[line width=3pt, color=red]
(v0) edge (u0)
(u1) edge (u)
;
\path
(u) edge (u0)
(v0) edge (u)
(u0) edge (u1)
;
\draw [dashed] (u1) to [bend right=90,looseness=2.5]  (u);
\node at (0,-3.25) {$M_1$};
\end{scope}
\begin{scope}[shift={(0,-3.16)},rotate=-60]
\hollowvertex(u0) at (-1,-1.73/3) {};
\hollowvertex(v0) at (1,-1.73/3) {};
\hollowvertex(u) at (1,-2-1.73/3) {};
\hollowvertex (u1) at (-1,-2-1.73/3) {};
\path[line width=3pt, color=red]
(v0) edge (u0)
(u1) edge (u)
;
\path
(u) edge (u0)
(v0) edge (u)
(u0) edge (u1)
;
\draw [dashed] (u1) to [bend right=90,looseness=2.5]  (u);
\node at (0,-3.25) {$M_1$};
\end{scope}
\end{scope}
\end{tikzpicture}
\caption{The graph $H_1$}
\label{H1}
\end{subfigure}
\caption{The graphs $G_1$ and $H_1$}
\label{G1H1}
\end{figure}
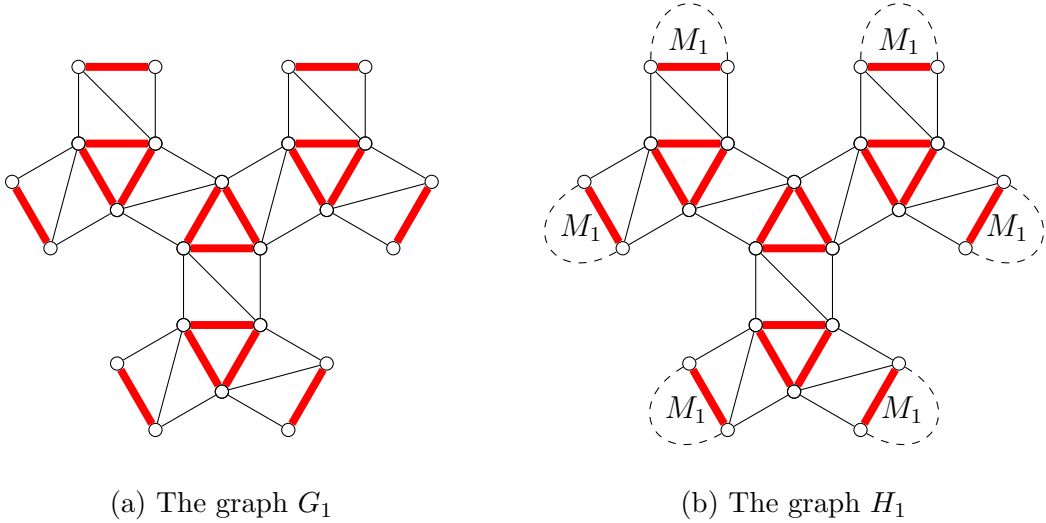

Now let $M_i$ be a sufficiently large maximal outerplanar graph with half of its vertices having degree $2$.  For a given integer $k$,  one can obtain such a graph $M_i$ of order $2k$ from an arbitrary maximal outerplanar graph $F$ of order $k$ as follows.  For every outer edge $e$ of $F$, add a new vertex $v_e$, and join $v_e$ to the endvertices of $e$. Colour one outer edge of $M_i$ red, and colour all other edges of $M_i$ black.  Construct the graph $H_i$ from $G_i$ by gluing a copy of $M_i$ (red-on-red) to every outer red edge of $G_i$ (see Figure~\ref{H1}).  We choose the graph $M_i$ in such a way that $|V(G_i)|=o(|V(M_i)|)$.  This way, if one chooses two vertices $u$ and $v$ at random from $H_i$, then almost surely neither belongs to $G_i$, and for sufficiently large $i$, they are in fact almost surely in different copies of $M_i$.

Whenever a trigon and a lozenge (or a lozenge and a copy of $M_i$) share a red edge in $H_i$, we say that they are {\em adjacent}.
If two vertices $u$ and $v$ of $H_i$ belong to distinct components of $H_i-E$, where $E$ is the set of black edges of a lozenge, then this lozenge is called a {\em connector lozenge} for $u$ and $v$.

We now consider an arbitrary orientation $D_i$ of $H_i$ for some $i\geq0$.  Suppose that vertices $u$ and $v$ are not in $G_i$, and are in distinct copies of $M_i$ (which is the case for almost all pairs of vertices asymptotically).  In this case, there must be at least two connector lozenges for $u$ and $v$, so $\theta(u,v)\leq 3$.  Suppose that $\theta(u,v)=3$.  We may assume, without loss of generality, that $\kappa(u,v)=2$ and $\kappa(v,u)=1$. So in every connector lozenge for $u$ and $v$, the independent black edges must be oriented away from $u$ towards $v$, and the third black edge must be oriented away from $v$ towards $u$.  In particular, a majority of the black edges in any connector lozenge between $u$ and $v$ must be oriented away from $u$ towards $v$.

We take a step back to offer some intuition at this point.  The deletion of a trigon from $H_i$ leaves three components $C_1, C_2$ and $C_3$. By the argument of the previous paragraph, there must be some pair of components, say $C_1$ and $C_2$, such that for every vertex  $u$ in $C_1$ and every vertex $v$ in $C_2$, we have $\theta(u,v)\leq 2$ in $D_i$. Figure \ref{orientation_blackedges_adj_w_trigon} shows an orientation of the black edges of the lozenges adjacent to a trigon. This orientation allows $\theta(u,v) =3$ for some $u \in V(C_i)$ where $i \in \{1, 2\}$  and $v \in V(C_3)$. However, this forces $\theta (u,v) =2$ for every $u \in V(C_1)$ and $v \in V(C_2)$. Most importantly, this happens at the initial trigon of our construction, and this means that a large proportion of pairs will have $\theta(u,v)\leq 2$.

\begin{figure}[htb]
\centering
\begin{tikzpicture}[scale=0.8]
\hollowvertex(u0) at (-1,-1.73/3) {};
\hollowvertex(v0) at (1,-1.73/3) {};
\hollowvertex(u) at (1,-2-1.73/3) {};
\hollowvertex (u1) at (-1,-2-1.73/3) {};
\path[line width=3pt, color=red]
(v0) edge (u0)
(u1) edge (u)
;
\path
(u0) edge[->-=0.75,> = triangle 45] (u)
(u) edge[->-=0.75,> = triangle 45] (v0)
(u1) edge[->-=0.75,> = triangle 45] (u0)
;
\draw [dashed] (u1) to [bend right=90,looseness=2.5]  (u);
\node at (0,-3.25) {$C_3$};
\begin{scope}[rotate=120]
\hollowvertex(u0) at (-1,-1.73/3) {};
\hollowvertex(v0) at (1,-1.73/3) {};
\hollowvertex(u) at (1,-2-1.73/3) {};
\hollowvertex (u1) at (-1,-2-1.73/3) {};
\path[line width=3pt, color=red]
(v0) edge (u0)
(u1) edge (u)
;
\path
(u) edge[->-=0.75,> = triangle 45] (u0)
(v0) edge[->-=0.75,> = triangle 45] (u)
(u0) edge[->-=0.75,> = triangle 45] (u1)
;
\draw [dashed] (u1) to [bend right=90,looseness=2.5]  (u);
\node at (0,-3.25) {$C_2$};
\end{scope}
\begin{scope}[rotate=240]
\hollowvertex(u0) at (-1,-1.73/3) {};
\hollowvertex(v0) at (1,-1.73/3) {};
\hollowvertex(u) at (1,-2-1.73/3) {};
\hollowvertex (u1) at (-1,-2-1.73/3) {};
\path[line width=3pt, color=red]
(v0) edge (u0)
(u1) edge (u)
;
\path
(u) edge[->-=0.75,> = triangle 45] (u0)
(v0) edge[->-=0.75,> = triangle 45] (u)
(u0) edge[->-=0.75,> = triangle 45] (u1)
;
\draw [dashed] (u1) to [bend right=90,looseness=2.5]  (u);
\node at (0,-3.25) {$C_1$};
\end{scope}
\end{tikzpicture}
\caption{An orientation of the black edges of the lozenges adjacent to a trigon} \label{orientation_blackedges_adj_w_trigon}
\end{figure}
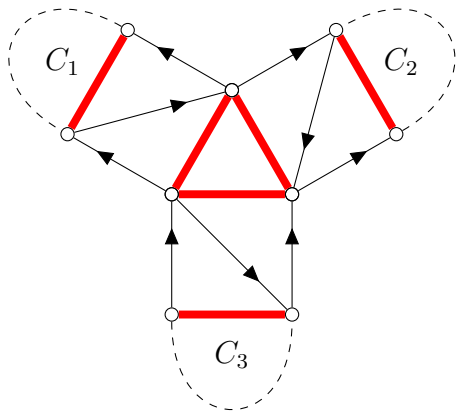

Define an auxiliary graph $A_i$ from $H_i$ as follows.  The vertex set of $A_i$ is the set of trigons of $H_i$ together with the set of copies of $M_i$ as subgraphs in $H_i$. Two vertices of $A_i$ are adjacent in $A_i$ if and only if they are adjacent to a common lozenge in $H_i$. (So edges in $A_i$ correspond exactly to lozenges in $H_i$.) The orientation $D_i$ of $H_i$ gives rise to an orientation $D_i^*$ of $A_i$ as follows: orient edge $uv$ in $A_i$ as $(u,v)$ if a majority of the black arcs in the corresponding lozenge of $H_i$ are directed away from $u$ towards $v$.

In order to bound the average connectivity of $D_i$, we now bound the average connectivity between leaves in the orientation $D_i^*$ of $A_i$.  There are $3\cdot 2^i$ leaves in $A_i$, and by a straightforward induction, one can prove that there are at most $2\cdot 4^i$ pairs of leaves in $D_i^*$ that are connected by a directed path.  Thus, if we pick two leaves of $D_i^*$ at random, then the probability that there is a directed path between them (and hence the probability that the connector lozenges between the corresponding copies of $M_i$ all have a majority of black edges oriented the same way) tends to at most $p=4/9$.

We now return to the orientation $D_i$ of $H_i$.  For each copy of $M_i$, we consider the set of vertices in $M_i$ that don't belong to $G_i$ (i.e., those that are not incident with the red edge).  Note that exactly half of these vertices have degree $2$, and hence the proportion of these vertices that can have $\theta$ value $3$ with some vertex in another copy of $M_i$ is at most $q=1/2$.  Thus, the average connectivity of $D_i$ is at most
\begin{align*}
\frac{3pq^2 + 2(1-pq^2)}{2} +o(1)&=1+\frac{pq^2}{2}+o(1)\leq 1+\tfrac{1}{18}+o(1).
\end{align*}
Since $D_i$ was an arbitrary orientation of $H_i$, we have
\[
\KBM(H_i)\leq 1 + \frac{1}{18}+o(1).
\]

Note that we have only shown an upper bound on $\KBM(H_i)$; i.e., we have not shown that $\KBM(H_i)\geq 1+\tfrac{1}{18}$.  However, for each $i$, we can choose $M_i$ so that we do in fact have
\[
\lim_{i\rightarrow\infty}\KBM(H_i)=1+\frac{1}{18}.
\]
We omit the details of explicitly describing both $M_i$ and an orientation of the resulting graph $H_i$, and then demonstrating a lower bound on the average connectivity of this orientation.  However, we note that an optimal orientation of $A_i$, which also maximizes the average connectivity between leaves, is easily obtained using the results of~\cite{HenningOellermann2004}.  An orientation of the black edges in the lozenges of $H_i$ can be ``lifted'' from this optimal orientation of $A_i$ as follows: if $(u,v)$ is an arc in $A_i$, then orient the independent black edges in the corresponding lozenge of $H_i$ away from $u$ towards $v$, and orient the last black edge in the corresponding lozenge of $H_i$ away from $v$ towards $u$.  This orientation of the lozenges more or less determines the orientations of the edges in the trigons.
\end{example}



\subsection{Orientations of $2$-Trees}

We begin with a lower bound for $\kbm(G)$ if $G$ is a $2$-tree.

\begin{theorem}
Let $G$ be a $2$-tree of order $n\geq 3$. Then
\[ \overline{\kappa}_{max}(G) \geq 1 + \frac{n-3}{n(n-1)}, \]
and this bound is sharp for all $n$.
\end{theorem}
\begin{proof}
We show by induction on $n$ that every $2$-tree $G$ of order $n$ has a
strong orientation $D$ with $K(D) \geq n^2-3$.
If $n=3$, then $G$ is a triangle, and orienting the edges of $K_3$ as
a directed cycle yields a digraph of total connectivity $6$, so the
statement holds for $n=3$.
Now let $G$ be a $2$-tree of order $n$. Then $G$ has a vertex $u$ such
that $G-u$ is a $2$-tree, and the neighbourhood of $u$ in $G$ consists
of two adjacent vertices, $v$ and $w$ say.
By our induction hypothesis, $G-u$ has a strong orientation $D'$ of total connectivity
at least $(n-1)^2-3$. We extend $D'$ to a strong orientation $D$ of $G$ by
orienting the edges $uv$ and $uw$ so that they form a directed $3$-cycle
together with $vw$. Then $D$ is strong. Hence $D$ contains $2(n-1)$ paths,
from $u$ to and from all vertices in $D'$. Furthermore, $D$ contains
a new path between $v$ and $w$ through $u$, that has no edges in common
with any path in $D'$. Hence $K(D) \geq K(D') + 2n-1 \geq (n-1)^2 - 3 +2n-1 = n^2 -3$.

For a given $n \ge 3$, let $G_n=K_2+{\overline{K}}_{n-2}$ (the join of $K_2$ and $\overline{K}_{n-2}$) and let $u,v$ be the vertices of degree $n-1$ in $G_n$. Let $D_n$ be an optimal orientation of $G_n$.  If $a,b$ is a pair of vertices of $G_n$ such that $\{a,b\} \ne \{u,v\}$, then $\theta(a,b) \le 2$. Moreover, $\theta(u,v) \le n-1$.  Hence $K(D_n) \le n(n-1) +n-3$ and thus $\kbm(G_n) \le 1+\frac{n-3}{n(n-1)}$. Thus the given bound is sharp.
\end{proof}

\smallskip

We also conjecture that if $G$ is a $2$-tree of order $n$ for which $\kbm(G)$ is largest, then $G$ is maximal outerplanar.
If this conjecture holds, then the results of Section \ref{MOPsection} tell us that $\kbm(G) \le \frac{3}{2} +o(1)$ for every $2$-tree $G$, and that this bound is asymptotically sharp.

\section{Conclusion}

The problem of finding the maximum average connectivity among all orientations of a graph was introduced in~\cite{HenningOellermann2004}, where the following asymptotically tight bounds for the maximum average connectivity of a tree $T$ were established:
\[\frac{2}{9} < \kbm(T) \le \frac{1}{2}.\]
In this paper we focused on finding bounds for $\kbm(G)$ and $\kbm(G)/\kb(G)$ for graphs $G$ belonging to certain classes that extend trees.

\medskip

We showed that if $G$ is a minimally $2$-connected graph, then \[1 \le \kbm(G) < \frac{5}{4},\]
and \[\frac{4}{9} \le \frac{\kbm(G)}{\kb(G)} < \frac{5}{8}.\]

\noindent We know that the lower bound for $\kbm(G)$ is sharp, but suspect that the upper bound can be improved.

\begin{problem}
Find an asymptotically sharp upper bound for $\KBM(G)$ if $G$ is a minimally $2$-connected graph.
\end{problem}

\noindent The bounds on the ratio have not yet been shown to be tight.

\begin{problem}
Determine asymptotically sharp upper and lower bounds for the ratio $\frac{\kbm(G)}{\kb(G)}$ if $G$ is a minimally $2$-connected graph.
\end{problem}

\noindent If the upper bound of $5/4$ for $\kbm(G)$ can be improved, this will also give rise to an improved upper bound for the ratio $\frac{\kbm(G)}{\kb(G)}$ of minimally $2$-connected graphs $G$.

\medskip

For every maximal outerplanar graph $G$, we proved that
\[\kbm(G) \le \frac{3}{2} +o(1),\]
and that this bound is asymptotically sharp.
For the lower bound, we propose the following.

\begin{conjecture}
For every maximal outerplanar graph $G$ of order at least $4$, we have
\[
\kbm(G)\ge 19/18.
\]
\end{conjecture}

\noindent We demonstrated that if this conjectured bound holds, then it is asymptotically sharp.

\medskip

For a graph $G$ in the class of $2$-trees (which contain the maximal outerplanar graphs), we showed that the bound $\kbm(G) \ge 1$ is asymptotically sharp. We propose the following conjecture.

\begin{conjecture}
If $G$ is a $2$-tree of order $n \ge 3$ for which $\kbm(G)$ is largest, then $G$ is maximal outerplanar.
\end{conjecture}
\noindent If this conjecture holds, then the upper bound given in Theorem \ref{theo:upper-bound-for-MOP} is also an upper bound for $\kbm(G)$ if $G$ is a $2$-tree.

\medskip

We observed that if $G$ is a graph, then $\kbm(G)/\kb(G) \le 1$, and we proved that this bound is asymptotically sharp. However, the following remains open.

\begin{problem}
 Does there exist a constant $c>0$ such that $\kbm(G)/\kb(G) \ge c$ for every (2-)connected graph $G$?
\end{problem}

\noindent For every tree $T$, it is known that $\kbm(T)/\kb(T) > 2/9$, and that this bound is asymptotically sharp.  For every $3$-connected cubic graph $G$, the fact that $\kbm(G)/\kb(G)\geq 1/3$ follows from Robbins' Theorem, and we demonstrated that this bound is asymptotically sharp.

\medskip

Very little is known about the global connectivity of optimal orientations of graphs. In particular, it would be interesting if the following could be answered.

\begin{problem}
Is every optimal orientation of every $2$-(edge-)connected graph strongly connected?
\end{problem}

\noindent Even the following weaker version of this problem remains open.

\begin{problem}
Does every $2$-(edge-)connected graph have a strong optimal orientation?
\end{problem}


\begin{thebibliography}{99}
\bibitem{BeinekeOellermannPippert2002} L.~W.~Beineke, O.~R.~Oellermann and R.~E.~Pippert, \emph{{The average connectivity of a graph}}, Discrete Math., \textbf{252} (2002), 31--45.

\bibitem{CasablancaMolOellermann2018}
R.~M.~Casablanca, L.~Mol and O.~R.~Oellermann, \emph{{The average connectivity of minimally
  $2$-connected graphs  and the average connectivity of minimally $2$-edge connected graphs}}, 2018. arXiv:1810.01972

\bibitem{Chvatal1973}
V.~Chv\'atal, \emph{Tough graphs and Hamiltonian circuits}, Discrete Math., \textbf{5} (1973), 215--228.


\bibitem{DankelmannOellermann2003} P.~Dankelmann, O.~R.~Oellermann, \emph{{On the average connectivity of a graph}}, Discrete Appl. Math., \textbf{129} (2003),  305--318.

\bibitem{Dirac1967}
G.~A.~Dirac, \emph{{Minimally $2$-connected graphs}}, J. Reine Angew. Math., \textbf{228} (1967), 204--216.

\bibitem{DurandDegevigney2012} O.~Durand de Gevigney, \emph{{On Frank's conjecture on $k$-connected orientations.}} arXiv: 12.12.4086, Dec. 2012.

\bibitem{HenningOellermann2004} M.~A.~Henning and O.~R.~Oellermann, \emph{{The average connectivity of a digraph}}, Discrete Appl. Math., \textbf{140} (2004), 143--153.

\bibitem{HenningOellermann2001}
M.~A.~Henning and O.~R.~Oellermann, \emph{{The average connectivity of regular multipartite tournaments.}}  Austral. J. Combin. \textbf{23} (2001), 101--113.

\bibitem{Mader1972}
W.~Mader, \emph{{Ecken vom Grad $n$ in minimalen $n$-fach zusammenh\"{a}ngenden
  Graphen}}, Arch. Math., \textbf{23} (1972), 219--224.

\bibitem{Mader1978}
W.~Mader, \emph{A reduction method for edge-connectivity in graphs}, Ann. Discr. Math. \textbf{3} (1978), 145--164.

\bibitem{Nash-williams1960} C.~St.~J.~A.~Nash-Williams, \emph{{On orientations, connectivity and odd- vertex-pairings in finite graphs}}, Canad. J. Math. \textbf{12} (1960), 555--567.

\bibitem{OellermannChapter2013}
O.~R. Oellermann, \emph{{Menger's Theorem}}, Topics in Structural Graph Theory
  (L.~W. Beineke and R.~J. Wilson, eds.), Cambridge University Press, 2013.

\bibitem{Plummer1968}
M.~D.~Plummer, \emph{{On Minimal Blocks}}, Trans. Amer. Math. Soc.,
  \textbf{134}({1}) (1968), 85--94.


\bibitem{Robbins1939} H.~E.~Robbins, \emph{{Questions, discussions, and notes: a theorem on graphs, with an application to a problem of traffic control}}, Amer. Math. Monthly \textbf{46} (1939), 281--283.

\bibitem{Thomassen1989} C.~Thomassen, \emph{{Configurations in graphs of large minimum degree, connectivity, or chromatic number}}, Combinatorial Mathematics: Proceedings of the Third International Conference, New York 1985, Ann. New York Acad. Sci., New York \textbf{555} (1989), 402--412.

\bibitem{Thomassen2014} C.~Thomassen, \emph{{Strongly $2$-connected orientations of graphs}}, J. Combin.Theory Ser B, \textbf{110} (2014), 67--78. 
\end{thebibliography}

\bigskip
\section*{Acknowledgements}
The authors wish to thank the Banff International Research Station for their support of the focussed research group number 18frg233, ``Measuring the Connectedness of Graphs and Digraphs'', during which most of the research for this manuscript was completed.

\end{document}